\documentclass{article}
\usepackage{subfig}
\usepackage[section] {placeins}
\usepackage{amsthm}
\usepackage{amssymb,amsmath}
\usepackage{anysize}
\usepackage{psfrag}
\newtheorem{thm}{Theorem}[section]
\newtheorem{def.}{Definition}[section]

\newtheorem{cor}{Corollary}[section]

\numberwithin{table}{section}
\begin{document}

\title{The minimum number of Fox colors modulo 13 is 5}
\author{Filipe Bento\\
        Department of Computer Science and Engineering
        \and
        Pedro Lopes\\
        Center for Mathematical Analysis, Geometry, and Dynamical Systems\\
        Department of Mathematics\\
        \and
        Instituto Superior T\'ecnico, University  of Lisbon\\
        Av. Rovisco Pais\\
        1049-001 Lisbon\\
        Portugal\\
        \texttt{filipe.s.bento@tecnico.ulisboa.pt, pelopes@math.tecnico.ulisboa.pt}\\
}
\date{October 04, 2015}
\maketitle

\begin{abstract}
In this article we show that if a knot diagram admits a non-trivial coloring modulo 13 then there is an equivalent diagram which can be colored with 5 colors. Leaning on known results, this implies that the minimum number of colors modulo 13 is 5.
\end{abstract}

\bigbreak


\bigbreak


\bigbreak

\section{Introduction}

\noindent

The Fox colorings of a knot or link (\cite{CFox}, exercise 6 on page 92) are the solutions of a system of linear homogeneous equations read off a diagram of the knot or link at issue. Arcs of the diagram are envisaged as algebraic variables and at each crossing of the diagram the equation ``twice the over-arc minus the under-arcs equals zero'' is read. The coefficient matrix of this system of equations is called {\bf coloring matrix} of the diagram under study. It has the following feature. Along each row there are exactly one $2$, two $-1$'s and the rest of the entries are $0$'s. It follows that the determinant of the coloring matrix is $0$. Furthermore, upon performance of a Reidemeister move on the diagram, the coloring matrix corresponding to the new diagram relates to the coloring matrix corresponding to the former diagram by elementary transformations on matrices. Thus the equivalence class of the coloring matrix for any diagram of the knot under study is an invariant. Let us choose for representative of this equivalence class the Smith Normal Form (SNF) of the coloring matrix. Since the determinant of these matrices is $0$, then one of the entries of the diagonal of the SNF is $0$ and this corresponds to the monochromatic solutions i.e., the solutions obtained by assigning the same color (number) to each arc of the diagram. Polichromatic solutions, also known as non-trivial solutions, are obtained if there is at least one more $0$ along the diagonal of the SNF. In general, especially in the case of knots and non-split links, this involves working over the modular integers for a specific prime modulus $p$. If our knot or link, $L$, admits non-trivial colorings over a modulus $p$, Harary and Kauffman, \cite{Frank},  introduced the minimum number of colors of $L$ mod $p$, notation $mincol_p (L)$ to be the minimum number of distinct colors it takes to assemble a non-trivial coloring mod $p$, the minimum being taken over all diagrams of the link at issue.

At this point we warn the reader that any knot or link considered in this article has non-zero determinant, the {\bf determinant} of the knot or link being the product of the entries of the diagonal of its SNF but the $0$ referred to above. As a matter of fact, a link with zero determinant is colorable modulo any prime which makes these links quite special and deserving a separate article.

There is  a  number of articles on the topic of minimum number of colors \cite{wchengetal, Ge5, JKL, kl, klgame, lm, lopesJKTR2015, nakamuranakanishisatoh, Oshiro, Saito, satoh}. In \cite{satoh}, Satoh developed a technique for finding the minimum number of colors over a fixed modulus but on an otherwise arbitrary situation. One considers a diagram equipped with a non-trivial coloring on the given modulus and using all available colors in all possible ways. The idea is then to remove one color at a time until one cannot remove any more colors. To remove one color one assumes it shows up in the diagram in all possible ways. Specifically, we assume it shows up as the color of a monochromatic crossing and devise a procedure to eliminate that color from this monochromatic crossing; we repeat the procedure for each monochromatic crossing bearing this color. These will be called the $\alpha$ instances. Then we assume the color at stake shows up at the over-arc of a polichromatic crossing and devise a procedure to eliminate it from the over-arc; and we repeat the procedure for all polichromatic crossings with this color on the over-arc. We call these the $\beta$ instances. Finally we assume the color shows up on an under-arc connecting two crossings, and devise a procedure to eliminate it and repeat it over all such situations. Here we distinguish two cases. If the adjacent over-arcs bear distinct colors we call them $\gamma $ instances; otherwise $\delta$ instances. In each of the $\alpha, \beta, \gamma, $ and $\delta$ instances, the procedure for eliminating the color consists of performing  Reidemeister moves (accompanied by the unique rearrangement of colors that yields a coloring in the new diagram) so that the color at issue is eliminated.  The transformations or sequences of Reidemeister moves that take care of $\alpha$ instances will be called $\alpha_i$'s and analogously for the other instances. If color $c$ has been successfully eliminated in each of the $\alpha, \beta, \gamma, $ and $\delta$ instances, one moves on to color $c'$ and iterates the procedure taking into consideration this time that color $c$ is no longer there (nor the colors previously removed). Satoh applied this technique successfully in \cite{satoh} to show that mod $5$ four colors suffice. Then Oshiro \cite{Oshiro} made the first impressive use of this technique by eliminating a string of $3$ colors mod $7$ thus showing that mod $7$, $4$ colors suffice. Using the same technique, Cheng et al \cite{wchengetal} proved that at most $6$ colors are needed mod $11$, and Nakamura et al \cite{nakamuranakanishisatoh} proved further that $5$ is the minimum number of colors for any knot or link admitting non-trivial $11$-colorings. In the current article we apply Satoh's technique to prove the following result.

\begin{thm}\label{thm:main}
For any knot or link $($with non-zero determinant$)$ admitting non-trivial colorings modulo $13$, there is a  diagram of it equipped with a non-trivial coloring modulo $13$ using $5$ colors.
\end{thm}
\begin{cor}
If $L$ is a knot or link in the conditions of Theorem \ref{thm:main}, then
\[
mincol_{13}\, L = 5.
\]
Furthermore, since there is essentially one set of $5$ colors modulo $13$ which can color a non-trivial coloring, there is a Universal $13$-Minimal Sufficient Set of Colors, in the sense of $\cite{lopesJKTR2015}$.
\end{cor}
\begin{proof}
Since it is known (\cite{nakamuranakanishisatoh, lm, JKL}) that the minimum number of colors modulo $13$ has to be at least $5$ then the equality follows from the Theorem.

Furthermore, it is also known (\cite{Ge5}, see also \cite{lopesJKTR2015}) that there is only one equivalence class of colors of cardinality $5$ modulo $13$. Then the Theorem proves that any such coloring set of $5$ colors colors any diagram which admits a minimal coloring modulo $13$.
\end{proof}

We prove Theorem \ref{thm:main} as follows. The sequence of colors to be removed is $12, 11, 6, 3, 4, 8, 9, 2$. We organize the removal of these colors into three parts. In the first one, Section \ref{sec:1211}, we remove colors $12$ and $11$ for few transformations are needed to eliminate these colors, with very few exceptions. In the second part, Section \ref{sec:6348}, we remove colors  $6, 3, 4$ and $8$. Here a considerable number of transformations is needed. Finally, in the third part, Section \ref{subsec:92},  we remove colors $9$ and $2$. Here, due to the fact that many colors have already been removed, the parameters can take on but a few values. Thus few transformations are needed in general and the really exceptional cases require a different approach.
\bigbreak
{\bf Remark} Any equality is to be understood modulo $13$, in the rest of this article.

\subsection{Acknowledgements}\label{sect:ackn}

\noindent

P.L. acknowledges support from FCT (Funda\c c\~ao para a Ci\^encia e a Tecnologia), Portugal, through project FCT EXCL/MAT-GEO/0222/2012, ``Geometry and Mathematical Physics''.

\section{Part I: Elimination of colors $12$ and $11$.}\label{sec:1211}

\bigbreak

\bigbreak
In this Section we remove colors $12$ and $11$. For each of the instances $\alpha$, $\beta$, $\gamma$, and $\delta$ there is one or two transformations that take care of most of the cases. Our strategy here is to force the unwanted situations and then prove that either they cannot occur or they are taken care of by one of the alternative transformations. Here are a couple of illustrative examples. The first one is the elimination of monochromatic crossings, with color $c=12$. Since our working coloring is non-trivial and the knot or link has non-trivial determinant, then if there are monochromatic crossings with color $c=12$, one of them has an adjacent crossing which looks like the left-hand side of transformations $\alpha_1$ (or $\alpha_2$) or $\alpha_3$, where $a\neq 12$ (Figure \ref{fig:monocrossnew}). (We know $a\neq 12$ since otherwise the coloring would be trivial.) Since $2a+1$ is our current $2a-c$ (since we assumed $c=12$) then transformations $\alpha_1$ and $\alpha_3$ guarantee the elimination of this monochromatic crossing provided $2a+1\neq 12$ since $12$ is precisely the color we are trying to remove. So forcing the unwanted situation here is to state $2a+1=12$. But the only solution is $a=12$ which as discussed above cannot occur. So we  conclude that transformations $\alpha_1$ and $\alpha_3$ suffice for eliminating the monochromatic crossings with color $12$. Now for a more elaborate example. The elimination of color $c=12$ from an over-arc of a polichromatic crossing, assuming there are no more monochromatic crossings with color $c=12$ left. The situation is depicted on the left-hand side of the top row of Figure \ref{fig:overcross0}. Thus $a\neq 12$ and if $2a+1\neq 12$ and $3a+2\neq 12$ then transformation $\beta_1$ does the job. So forcing the unwanted situations here is to state $2a+1=12$ or $3a+2=12$. But the solution to either of these equations is $a=12$ which cannot occur. We then conclude that the elimination of polichromatic crossings whose over-arcs bear color $12$ is realized by transformation $\beta_1$. Considering the left-hand side of Figure \ref{fig:overcross0} again we could argue that we also do not have to consider values of $a$ such that $24-a=12$. The fact is that our analysis of $2a+1=12$ or $3a+2=12$ allows us to disregard such details, in this case. Furthermore, it should be noted that in both situations we assumed (as we should) that $12$ was the first color to be removed i.e., no other color had already been removed. Subsequent cases might not be so simple. We now proceed  with the analysis of the other instances without this much detail.

For each of the $\alpha$, $\beta$, $\gamma$ and $\delta$ instances,  we will treat the elimination of $12$ and $11$; it will be clear to the reader that this is equivalent to treating the elimination of $12$ in the  $\alpha, \beta, \gamma, \delta$ instances followed by treating the elimination of $11$.
\begin{enumerate}
\item The $\alpha$ instance.

In order to remove color $c=12$ and $c=11$ from a monochromatic crossing, consider the transformations $\alpha_1$ in Figure \ref{fig:monocrossnew}.
\begin{figure}[!ht]
	\psfrag{alpha1}{\huge$\alpha_1$}
	\psfrag{alpha2}{\huge$\alpha_2$}
	\psfrag{alpha3}{\huge$\alpha_3$}
	\psfrag{alpha1'}{\huge$\alpha_1'$}
	\psfrag{a}{\huge$a$}
	\psfrag{c}{\huge$c$}
	\psfrag{2a-c}{\huge$2a-c$}
	\psfrag{2c-a}{\huge$2c-a$}
	\psfrag{3c-2a}{\huge$3c-2a$}
	 \centerline{\scalebox{.4}{\includegraphics{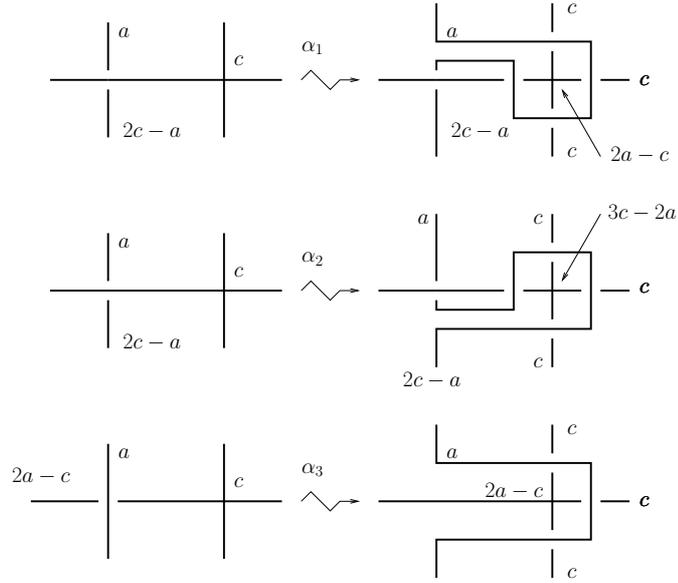}}}
	\caption{Transformation $\alpha_1$, $\alpha_2$, and $\alpha_3$. The $c$ ($2a-c$ or $3c-2a$) close to the crossing means this crossing is monochromatic with the indicated color. The details of these crossings are not specified for they do not matter for the argument. The following remark is in oreder for $\alpha_3$. Either $2a-c$ has already been removed or is currently being removed and the left hand-side of $\alpha_3$ will never materialize in our diagram. Or $2a-c$ is still available and $\alpha_3$ trivially eliminates the monochromatic crossing colored with color $c$.}\label{fig:monocrossnew}
\end{figure}
\begin{enumerate}
\item $c=12$: ditto.

\item $c=11=-2$

 Here the unwanted situations are $2a-11=11$ or $2a-11=12$. In the first case we have $a=11$ which cannot occur. In the second case we have $2a=23=10$ and so $a=5$. In this case transformation $\alpha_2$ takes care of the issue since $3c-2a=36-10=0$.
\end{enumerate}

\item The $\beta$ instance.

Consider transformation $\beta_1$ in Figure \ref{fig:overcross0}.

\begin{figure}[!ht]
	\psfrag{a}{\huge$a$}
	\psfrag{c}{\huge$c$}
	\psfrag{beta1}{\huge$\beta_1$}
	\psfrag{beta2}{\huge$\beta_2$}
	\psfrag{beta2}{\huge$\beta_2$}
	\psfrag{2a-c}{\huge$2a-c$}
	\psfrag{2c-a}{\huge$2c-a$}
	\psfrag{3a-2c}{\huge$3a-2c$}
	\psfrag{3c-2a}{\huge$3c-2a$}
	\psfrag{4c-3a}{\huge$4c-3a$}
	\psfrag{5c-4a}{\huge$5c-4a$}
	\psfrag{6c-5a}{\huge$6c-5a$}
	 \centerline{\scalebox{.4}{\includegraphics{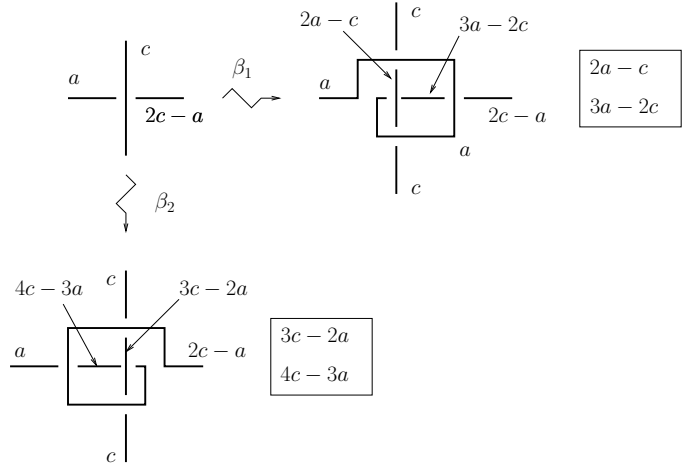}}}
	\caption{Transformations $\beta_1$ and $\beta_2$.}\label{fig:overcross0}
\end{figure}
\begin{enumerate}
\item $c=12$: ditto.

\item $c=11$.

Then if either $2a-11=11$ or $3a-22=11$ then $a=11$ which contradicts the working hypothesis. So let us look into  $2a-11=12$ and/or $3a-22=12$. The solutions are $a=5$ and $a=7$. Then $2c-a = 4$ and $2c-a = 2$, respectively, and transformation $\beta_2$ in Figure \ref{fig:overcross0} takes care of these issues.
\end{enumerate}

\item The $\gamma$ instance.

Consider transformation $\gamma_1$ in Figure \ref{fig:diffunder1bis}.

\begin{figure}[!ht]
	\psfrag{gamma1}{\huge$\gamma_1$}
	\psfrag{gamma2}{\huge$\gamma_2$}
	\psfrag{a}{\huge$a$}
	\psfrag{b}{\huge$b$}
	\psfrag{c}{\huge$c$}
	\psfrag{2b-c}{\huge$2b-c$}
	\psfrag{2b-a}{\huge$2b-a$}
	\psfrag{2a-b}{\huge$2a-b$}
	\psfrag{2a-c}{\huge$2a-c$}
	\psfrag{2b-2a+c}{\huge$2b-2a+c$}
	\psfrag{2a-2b+c}{\huge$2a-2b+c$}
	 \centerline{\scalebox{.4}{\includegraphics{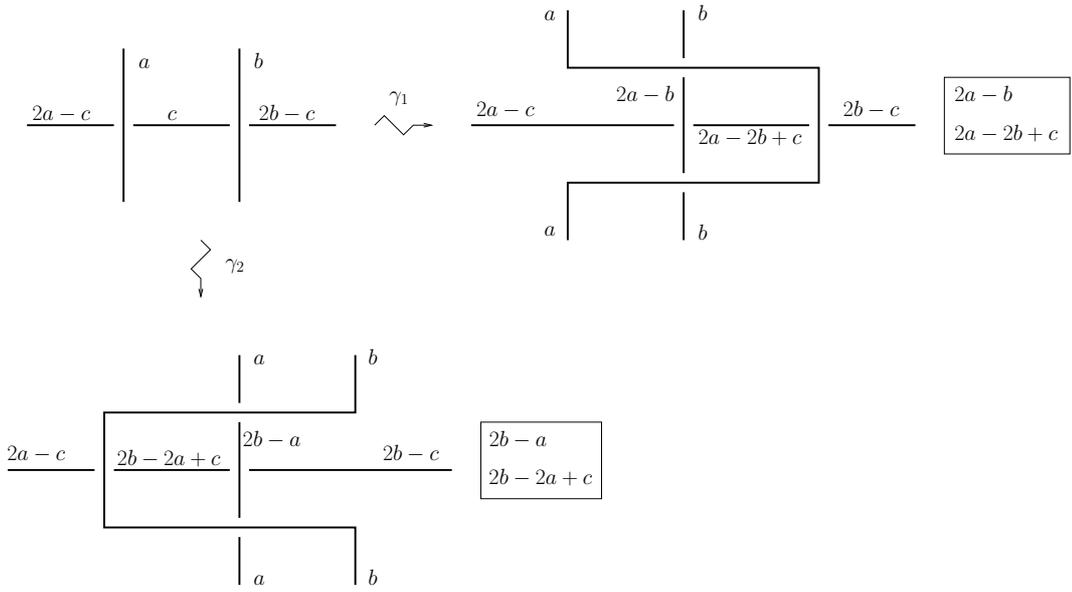}}}
	\caption{Transformations  $\gamma_1$ and $\gamma_2$.}\label{fig:diffunder1bis}
\end{figure}
\begin{enumerate}
\item $c=12$

Then $2a-2b+c=2a-2b-1$ and if $2a-2b-1=-1$ then $2(a-b)=0$ and so $a=b$ which is contrary to the assumptions.

If $2a-b=-1$ then $b=2a+1$ and let us now see the implications of this equality when considering transformation $\gamma_2$ (Figure \ref{fig:diffunder1bis}). Then $2b-a = 2(2a+1)-a=3a+2$. If $3a+2=12$ then $a=-1$ which is contrary to our assumptions. Also, $2b-2a-1=2a+1$ and if $2a+1=12$ then $a=12$ which is contrary to our assumptions. We conclude that, for $c=12$, transformation $\gamma_2$ solves the issues that are not solved by transformation $\gamma_1$.

\begin{figure}[!ht]
	\psfrag{gamma4}{\huge$\gamma_4$}
	\psfrag{gamma3}{\huge$\gamma_3$}
	\psfrag{a}{\huge$a$}
	\psfrag{b}{\huge$b$}
	\psfrag{c}{\huge$c$}
	\psfrag{2b-c}{\huge$2b-c$}
	\psfrag{2a-c}{\huge$2a-c$}
	\psfrag{3a-2c}{\huge$3a-2c$}
	\psfrag{3b-2c}{\huge$3b-2c$}
	\psfrag{4a-3c}{\huge$4a-3c$}
	\psfrag{4b-3c}{\huge$4b-3c$}
	\psfrag{4b-2c-a}{\huge$4b-2c-a$}
	\psfrag{4a-2c-b}{\huge$4a-2c-b$}
	\psfrag{4a-2b-c}{\huge$4a-2b-c$}
	\psfrag{4b-2a-c}{\huge$4b-2a-c$}
	 \centerline{\scalebox{.4}{\includegraphics{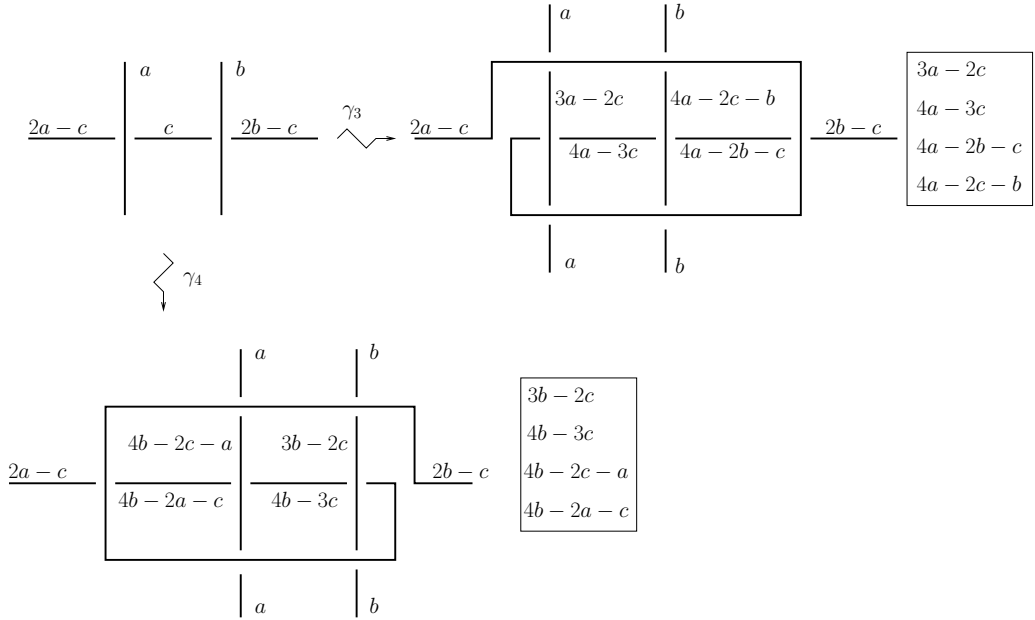}}}
	\caption{Transformations $\gamma_3$ and $\gamma_4$.}\label{fig:diffunder0bis}
\end{figure}

\item $c=11$

If $2a-b=-1$ then $b=2a+1$.

\begin{enumerate}
\item Then $2b-a=3a+2$ and if $3a+2=-1$ then $a=-1$ which contradicts the assumptions; if $3a+2=-2$ then $a=3$ and $b=2a+1=7$. This $(a, b)=(3, 7)$ case is dealt with with transformation $\gamma_3$ (Figure \ref{fig:diffunder0bis}).
\item On the other hand, $2b-2a-2=2a$. If $2a=-2$ then $a=-1$ contrary to the assumptions. If $2a=12$ then $a=6$ and so $b=2a+1=0$. The case $(a, b)=(6, 0)$ is dealt with with transformation $\gamma_3$ (Figure \ref{fig:diffunder0bis}).
\end{enumerate}

If $2a-b=-2$ then $b=2a+2$

\begin{enumerate}
\item Then $2b-a=3a+4$. So if $3a+4=-2$ then $a=-2$, contrary to the assumptions. If $3a+4=-1$ then $a=7$ and $b=2a+2=3$. This $(a, b)=(7, 3)$ case is dealt with with transformation $\gamma_4$ (Figure \ref{fig:diffunder0bis}).
\item On the other hand, $2b-2a-2=2a+2$ and $2a+2$ cannot take on $12$ or $11$ since it is part of the left-hand side of the transformations $\gamma_i$.
    \end{enumerate}

Now for $2a-2b+c = 2a-2b-2$. If $2a-2b-2=-2$ then $a=b$ which is contrary to the assumptions. If $2a-2b-2=-1$ then $a=b+7$.

\begin{enumerate}
\item Then $2b-a = b-7$. If $b-7 = -2$ then $b=5$ so that $2b+2=12$ which is contrary to the assumptions. If $b-7=-1$ then $b=6$ so that $a=b+7=0$. The case $(a, b)=(0, 6)$ is dealt with with transformation $\gamma_4$ (Figure \ref{fig:diffunder0bis}).
\item On the other hand, $2b-2a-2=-14-2=10$ which does not raise any objections.
\end{enumerate}

\item The $\delta$ instance.

Consider transformation $\delta_1$ in Figure \ref{fig:sameunder0}.
\begin{figure}[!ht]
	\psfrag{delta1}{\huge$\delta_{1}$}
	\psfrag{delta3}{\huge$\delta_3$}
	\psfrag{a}{\huge$a$}
	\psfrag{c}{\huge$c$}
	\psfrag{2a-c}{\huge$2a-c$}
	\psfrag{3a-2c}{\huge$3a-2c$}
	\psfrag{4a-3c}{\huge$4a-3c$}
	\psfrag{7a-6c}{\huge$7a-6c$}
	\psfrag{6a-5c}{\huge$6a-5c$}
	\psfrag{5a-4c}{\huge$5a-4c$}
\centerline{\scalebox{.4}{\includegraphics{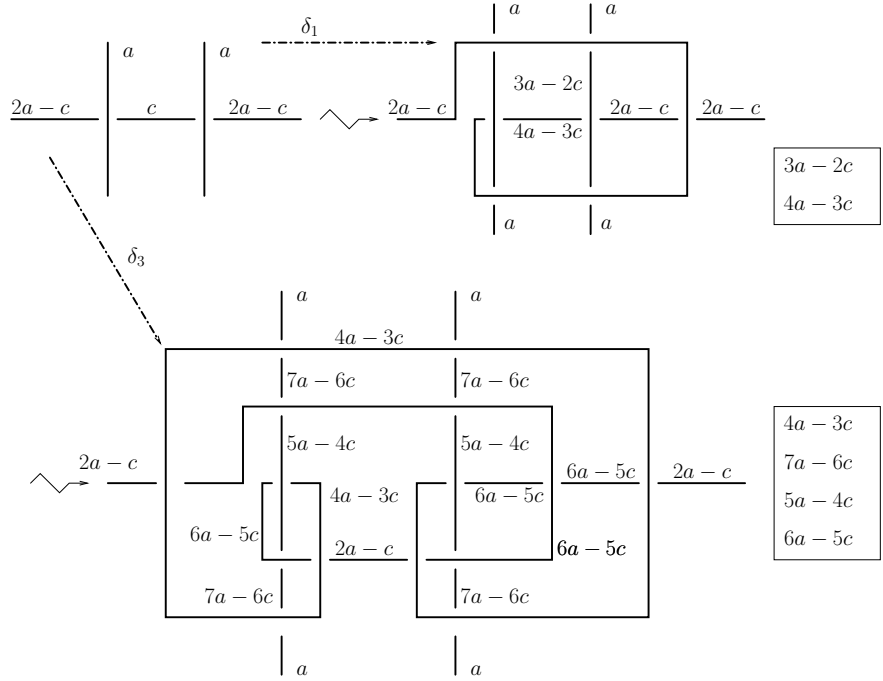}}}
	\caption{Transformations $\delta_1$ and $\delta_3$.}\label{fig:sameunder0}
\end{figure}

\begin{enumerate}

\item $c=12$.

Then $3a-2c=3a+2$ and if $3a+2=-1$ then $a=-1$ which contradicts the assumption. Also $4a-3c=4a+3$ and if $4a+3=-1$ then $a=-1$ which contradicts the assumption.

\item $c=11$.

Then $3a-2c=3a+4$. If $3a+4=-2$ then $a=-2$ which contradicts the assumptions. If $3a+4=-1$ then $a=7$ and this case is dealt with with transformation $\delta_3$ (Figure \ref{fig:sameunder0}).

Also, $4a-3c=4a+6$. If $4a+6=-2$ then $a=-2$ which contradicts the assumptions. If $4a+6=-1$ then $a=8$ and this case is dealt with with transformation $\delta_2$ (Figure \ref{fig:sameunder3}).
\end{enumerate}

\begin{figure}[!ht]
	\psfrag{delta1}{\huge$\delta_{1}$}
	\psfrag{delta2}{\huge$\delta_{2}$}
	\psfrag{a}{\huge$a$}
	\psfrag{c}{\huge$c$}
	\psfrag{2a-c}{\huge$2a-c$}
	\psfrag{3a-2c}{\huge$3a-2c$}
	\psfrag{5c-4a}{\huge$5c-4a$}
	\psfrag{6c-5a}{\huge$6c-5a$}
	\psfrag{3c-2a}{\huge$3c-2a$}
	\psfrag{7c-6a}{\huge$7c-6a$}
	 \centerline{\scalebox{.4}{\includegraphics{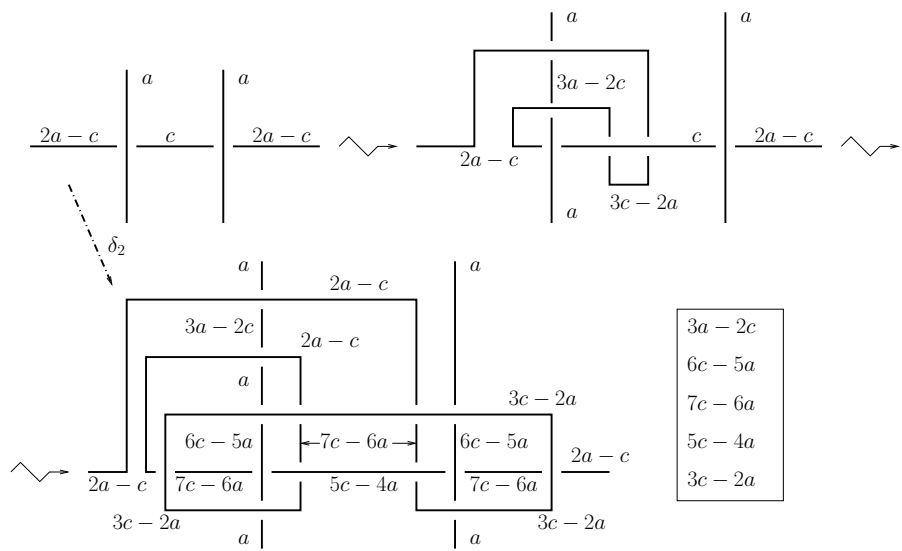}}}
	\caption{Transformation $\delta_2$.}\label{fig:sameunder3}
\end{figure}

\end{enumerate}
\end{enumerate}



\bigbreak

\bigbreak

\bigbreak

\bigbreak

\section{Part II: Elimination of colors $6, 3, 4$ and $8$.}\label{sec:6348}

\noindent

In this Section we eliminate colors $6, 3, 4$ and $8$. Since we now have already removed colors $12$ and $11$, our technique of forcing the unwanted situations would require dealing with a significant number of cases. In addition to that, the number of transformations needed is larger than before, as we noted in the course of our work. Therefore we change the strategy and simply supply a list of extra transformations for each of the $\beta$, $\gamma$, and $\delta$ instances (since for the $\alpha$ instances we do not need any more transformations) and display tables from which the relevant information can be read off. These tables have $2$ rows for the $\alpha, \beta,$ and $\delta$ instances. For these instances, each of these tables concern the elimination of a color $c$ and there is only one other parameter, $a$, which can take on the values of the colors still available. The first row displays the different values $a$ can take on; the second row displays the index of the transformation that settles the issue corresponding to the values of $a$ right above this index. Here is one illustrative example. In Table \ref{Ta:6beta}, the fourth entry in the top row is $3$ and the fourth entry in the bottom row is $1$. This tells us that the elimination of color $6$ from a polichromatic crossing whose over-arc bears color $6$ and one of the under-arcs bears color $3$ is realized with transformation $\beta_1$ (see Figure \ref{fig:overcross0}). In the same Table, the first entry in the top row is $0$ and the first entry in the bottom row is $X$. This tells us that a polichromatic crossing whose over-arc bears color $6$ and one of the under-arcs bears color $0$ cannot occur since the other under-arc would then bear color $2\cdot 6 - 0 = 12$ which has already been removed. There is a slightly different interpretation of $X$'s in the tables for the $\alpha$ transformations. Since $\alpha_3$ is trivially accomplished, we reserve the $X$'s for the situations involving the $\alpha_1$ or $\alpha_2$ transformations where $2c-a$ has already been removed or is currently being removed.

Now for the Tables concerning the $\gamma$ instances. These instances involve the elimination of a color $c$ in a situation involving two parameters, $a$ and $b$. These Tables should then be regarded as stacks of two consecutive rows. For the first pair of rows, the first row has to do with the first value parameter $a$ can take on; the different columns correspond to the different values parameter $b$ can take on subject to the given value of parameter $a$ (note that $a\neq b$ for the $\gamma$ transformations). The second row of this first pair of rows displays, in general, the indices of the transformations that settle the issue corresponding to the the value of $b$ right above it along with the given value of $a$. Analogously for the other pairs of rows. An analogous interpretation applies to the $X$'s as before. They correspond to $2a-c$ or $2b-c$ taking on colors which have already been eliminated or are being eliminated.

The reader will have to check that the linear combinations associated to a given transformation do not evaluate to one of the colors already eliminated or being eliminated. This seems preferable to obstructing the article with evaluating all these expressions - which the reader would still have to check.


\begin{figure}[!ht]
	\psfrag{a}{\huge$a$}
	\psfrag{c}{\huge$c$}
	\psfrag{2a-c}{\huge$2a-c$}
	\psfrag{2c-a}{\huge$2c-a$}
	\psfrag{3a-2c}{\huge$3a-2c$}
	\psfrag{4c-3a}{\huge$4c-3a$}
	\psfrag{5c-4a}{\huge$5c-4a$}
	\psfrag{6c-5a}{\huge$6c-5a$}
	\psfrag{9a-8c}{\huge$9a-8c$}
	\psfrag{5a-4c}{\huge$5a-4c$}
	\psfrag{3c-2a}{\huge$3c-2a$}
	\psfrag{9c-8a}{\huge$9c-8a$}
	\psfrag{8c-7a}{\huge$8c-7a$}
	\psfrag{16c-15a}{\huge$16c-15a$}
	\psfrag{12c-11a}{\huge$12c-11a$}
	\psfrag{22c-21a}{\huge$22c-21a$}
	\psfrag{15c-14a}{\huge$15c-14a$}
	\psfrag{0}{\huge$0$}
	\psfrag{1}{\huge$1$}
	\psfrag{2}{\huge$2$}
	\psfrag{3}{\huge$3$}
	\psfrag{4}{\huge$4$}
	\psfrag{5}{\huge$5$}
	\psfrag{6}{\huge$6$}
	\psfrag{11}{\huge$\mathbf{11}$}
	 \centerline{\scalebox{.4}{\includegraphics{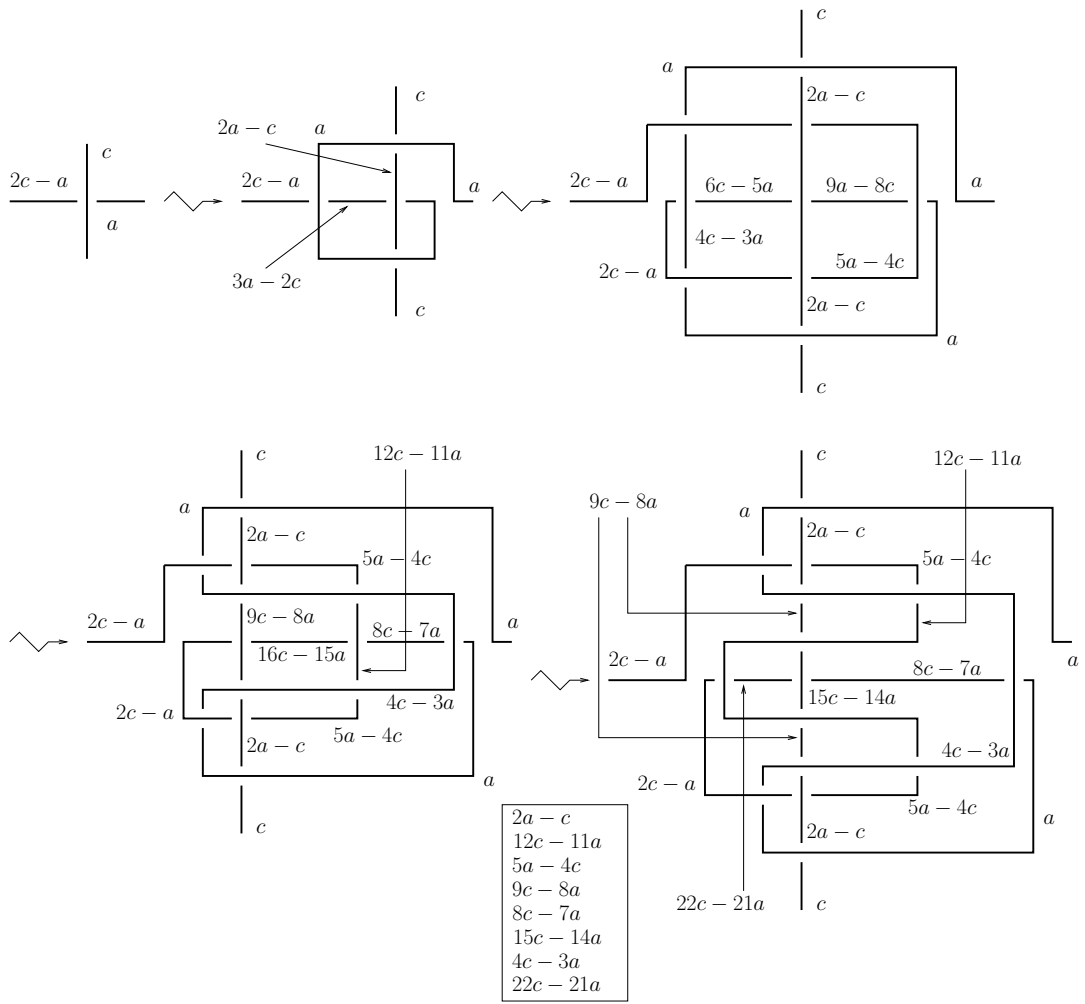}}}
	\caption{Transformation $\beta_3$: from the leftmost diagram in the upper row to the rightmost diagram in the lower row.}\label{fig:overcross15}
\end{figure}










\begin{figure}[!ht]
	\psfrag{gamma5}{\huge$\gamma_5$}
	\psfrag{gamma6}{\huge$\gamma_6$}
	\psfrag{a}{\huge$a$}
	\psfrag{b}{\huge$b$}
	\psfrag{c}{\huge$c$}
	\psfrag{2b-c}{\huge$2b-c$}
	\psfrag{2b-a}{\huge$2b-a$}
	\psfrag{2b-2a+c}{\huge$2b-2a+c$}
	\psfrag{2a-b}{\huge$2a-b$}
	\psfrag{2a-2b+c}{\huge$2a-2b+c$}
	\psfrag{2a-c}{\huge$2a-c$}
	\psfrag{2c-a}{\huge$2c-a$}
	\psfrag{3b-2a}{\huge$3b-2a$}
	\psfrag{4b-4a+c}{\huge$4b-4a+c$}
	\psfrag{3a-2b}{\huge$3a-2b$}
	\psfrag{4a-4b+c}{\huge$4a-4b+c$}
	 \centerline{\scalebox{.4}{\includegraphics{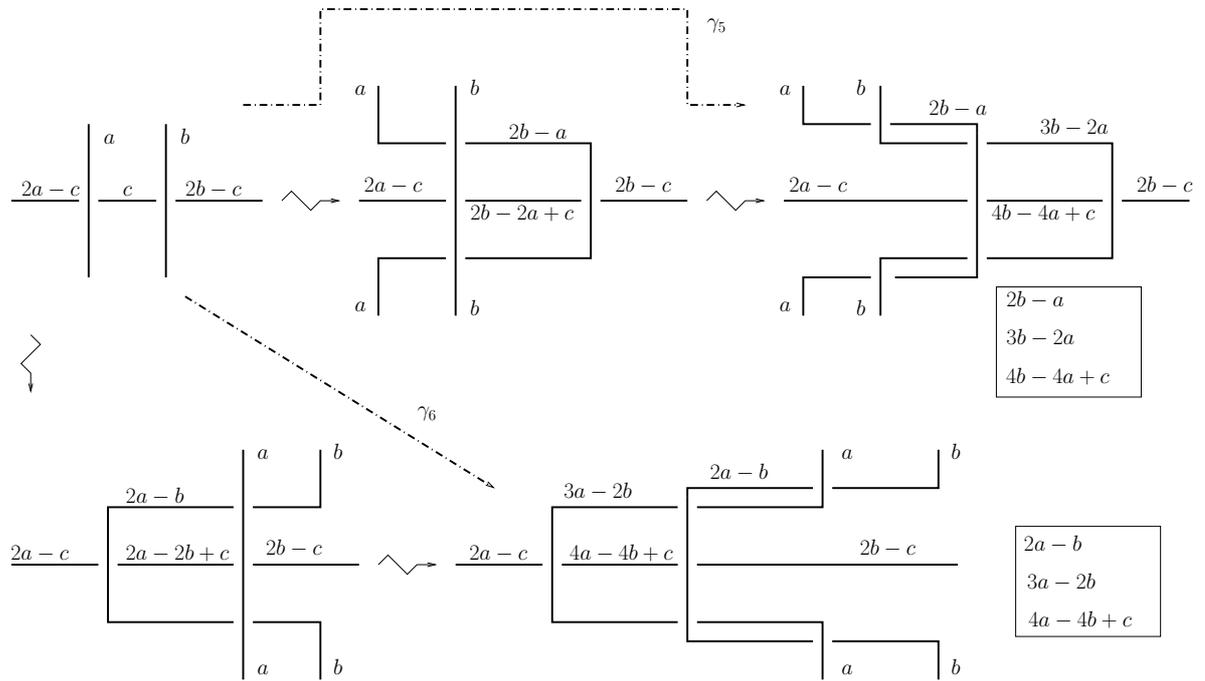}}}
	\caption{Transformations $\gamma_5$ and $\gamma_6$.}\label{fig:diffunder2}
\end{figure}



\begin{figure}[!ht]
	\psfrag{gamma7}{\huge$\gamma_7$}
	\psfrag{gamma6}{\huge$\gamma_6$}
	\psfrag{a}{\huge$a$}
	\psfrag{b}{\huge$b$}
	\psfrag{c}{\huge$c$}
	\psfrag{2b-c}{\huge$2b-c$}
	\psfrag{2b-a}{\huge$2b-a$}
	\psfrag{2a-b}{\huge$2a-b$}
	\psfrag{2a-2b+c}{\huge$2a-2b+c$}
	\psfrag{2a-2c+b}{\huge$2a-2c+b$}
	\psfrag{2a-3c+2b}{\huge$2a-3c+2b$}
	\psfrag{2a-c}{\huge$2a-c$}
	\psfrag{2c-a}{\huge$2c-a$}
	\psfrag{3a-2c}{\huge$3a-2c$}
	\psfrag{4a-3c}{\huge$4a-3c$}
	\psfrag{4a-2c-b}{\huge$4a-2c-b$}
	\psfrag{4a-2b-c}{\huge$4a-2b-c$}
	 \centerline{\scalebox{.4}{\includegraphics{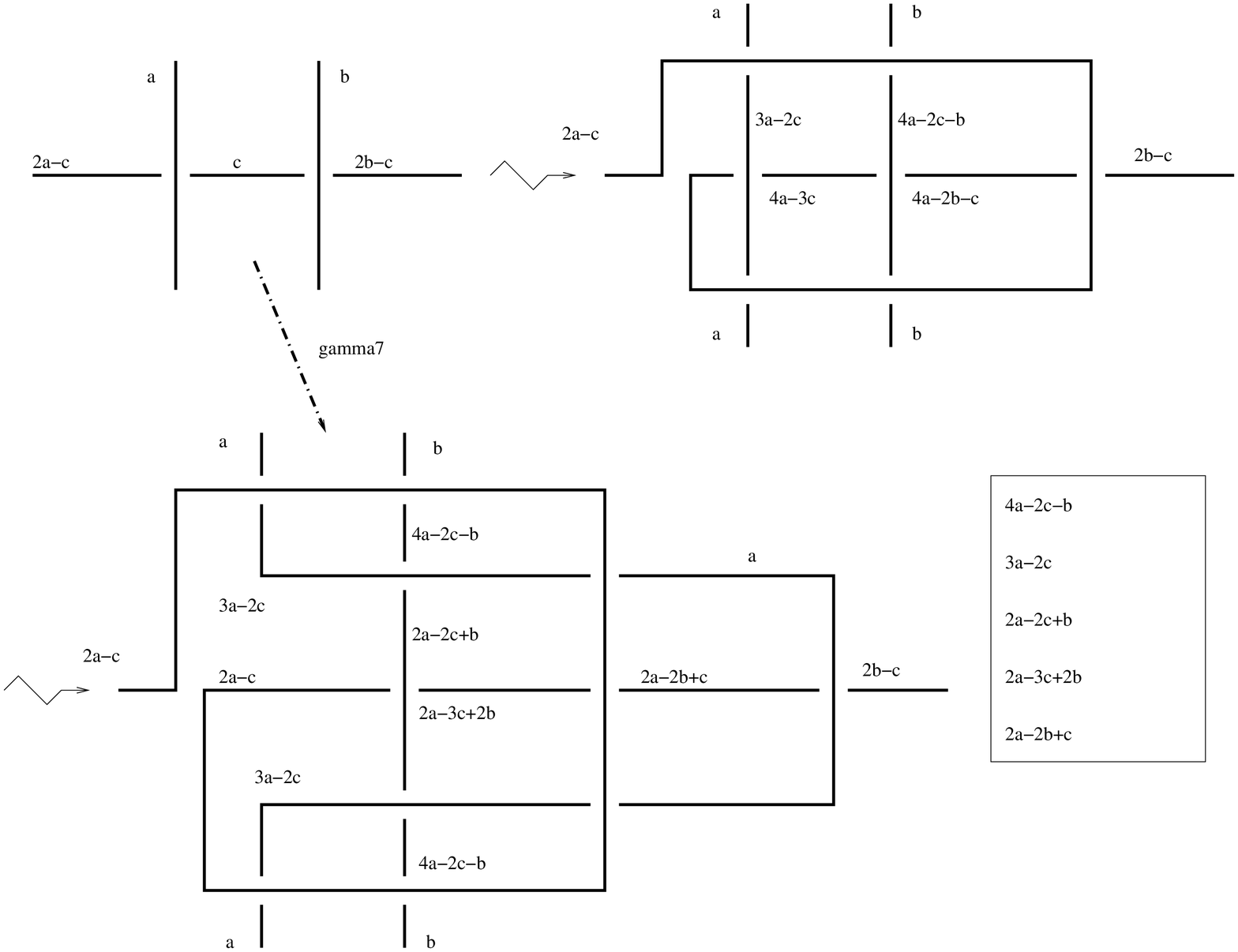}}}
	\caption{Transformation  $\gamma_7$.}\label{fig:diffunder7}
\end{figure}




\begin{figure}[!ht]
	\psfrag{gamma19}{\huge$\gamma_{10}$}
	\psfrag{gamma6}{\huge$\gamma_6$}
	\psfrag{a}{\huge$a$}
	\psfrag{b}{\huge$b$}
	\psfrag{c}{\huge$c$}
	\psfrag{2b-c}{\huge$2b-c$}
	\psfrag{2a-c}{\huge$2a-c$}
	\psfrag{4a-3c}{\huge$4a-3c$}
	\psfrag{3a-2c}{\huge$3a-2c$}
	\psfrag{4a-b-2c}{\huge$4a-b-2c$}
	\psfrag{4a-2b-c}{\huge$4a-2b-c$}
	\psfrag{6a-2b-3c}{\huge$6a-2b-3c$}
	\psfrag{8a-4b-3c}{\huge$8a-4b-3c$}
	\psfrag{5a-2b-2c}{\huge$5a-2b-2c$}
	\psfrag{10a-6b-3c}{\huge$10a-6b-3c$}
	\psfrag{4b-2a-c}{\huge$4b-2a-c$}
\centerline{\scalebox{.4}{\includegraphics{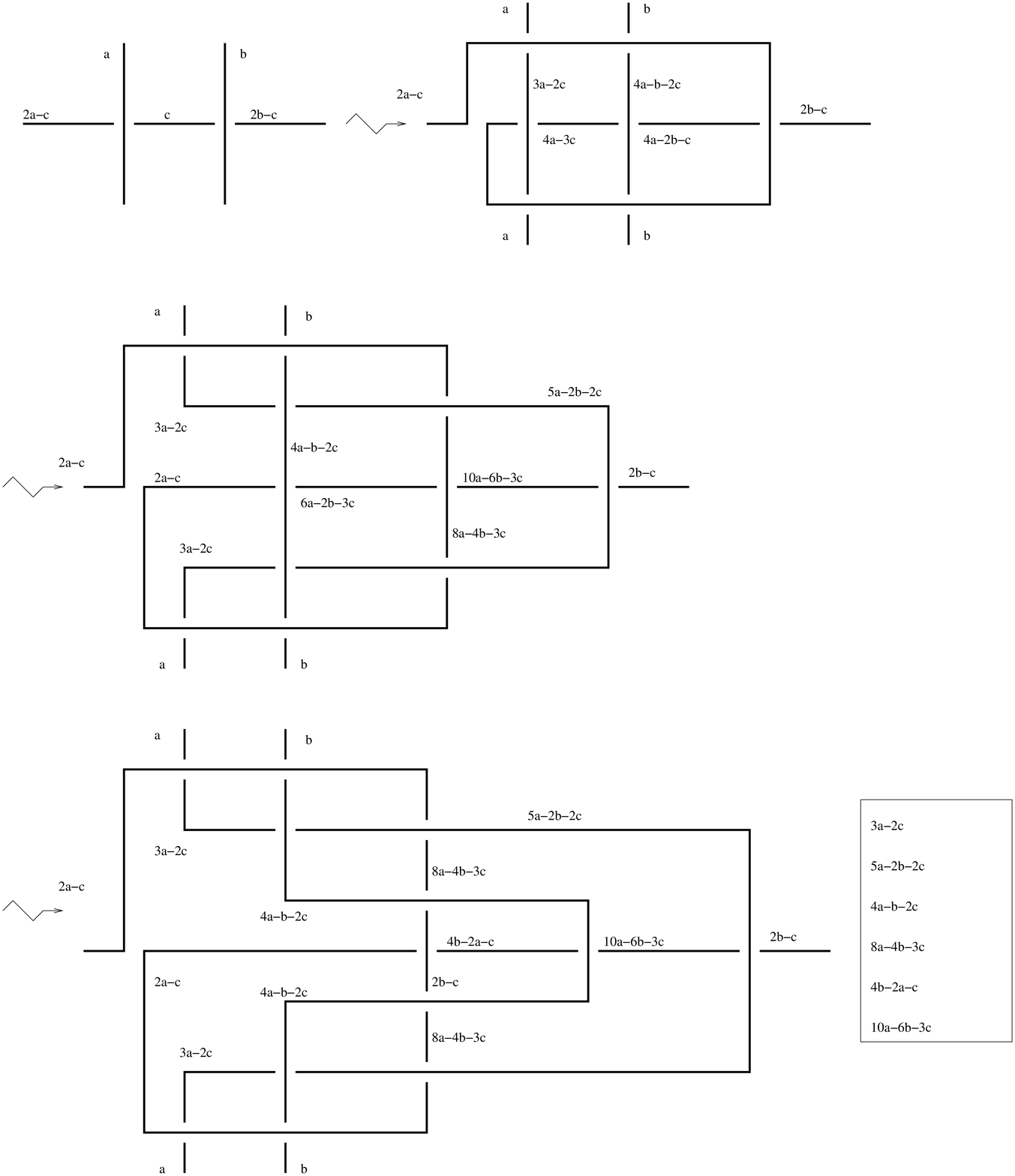}}}
	\caption{Transformation $\gamma_{8}$.}\label{fig:diffunder77}
\end{figure}

\begin{figure}[!ht]
	\psfrag{gamma15}{\huge$\gamma_{9}$}
	\psfrag{gamma6}{\huge$\gamma_6$}
	\psfrag{a}{\huge$a$}
	\psfrag{b}{\huge$b$}
	\psfrag{c}{\huge$c$}
	\psfrag{2b-c}{\huge$2b-c$}
	\psfrag{8b-4c-3a}{\huge$8b-4c-3a$}
	\psfrag{2a-c}{\huge$2a-c$}
	\psfrag{3b-2c}{\huge$3b-2c$}
	\psfrag{4b-3c}{\huge$4b-3c$}
	\psfrag{4a-3c}{\huge$4a-3c$}
	\psfrag{6b-3c-2a}{\huge$6b-3c-2a$}
	\psfrag{5b-2c-2a}{\huge$5b-2c-2a$}
	\psfrag{4b-2c-a}{\huge$4b-2c-a$}
	\psfrag{4b-2a-c}{\huge$4b-2a-c$}
	\psfrag{4a-2b-c}{\huge$4a-2b-c$}
	\psfrag{11b-6c-4a}{\huge$11b-6c-4a$}
\centerline{\scalebox{.4}{\includegraphics{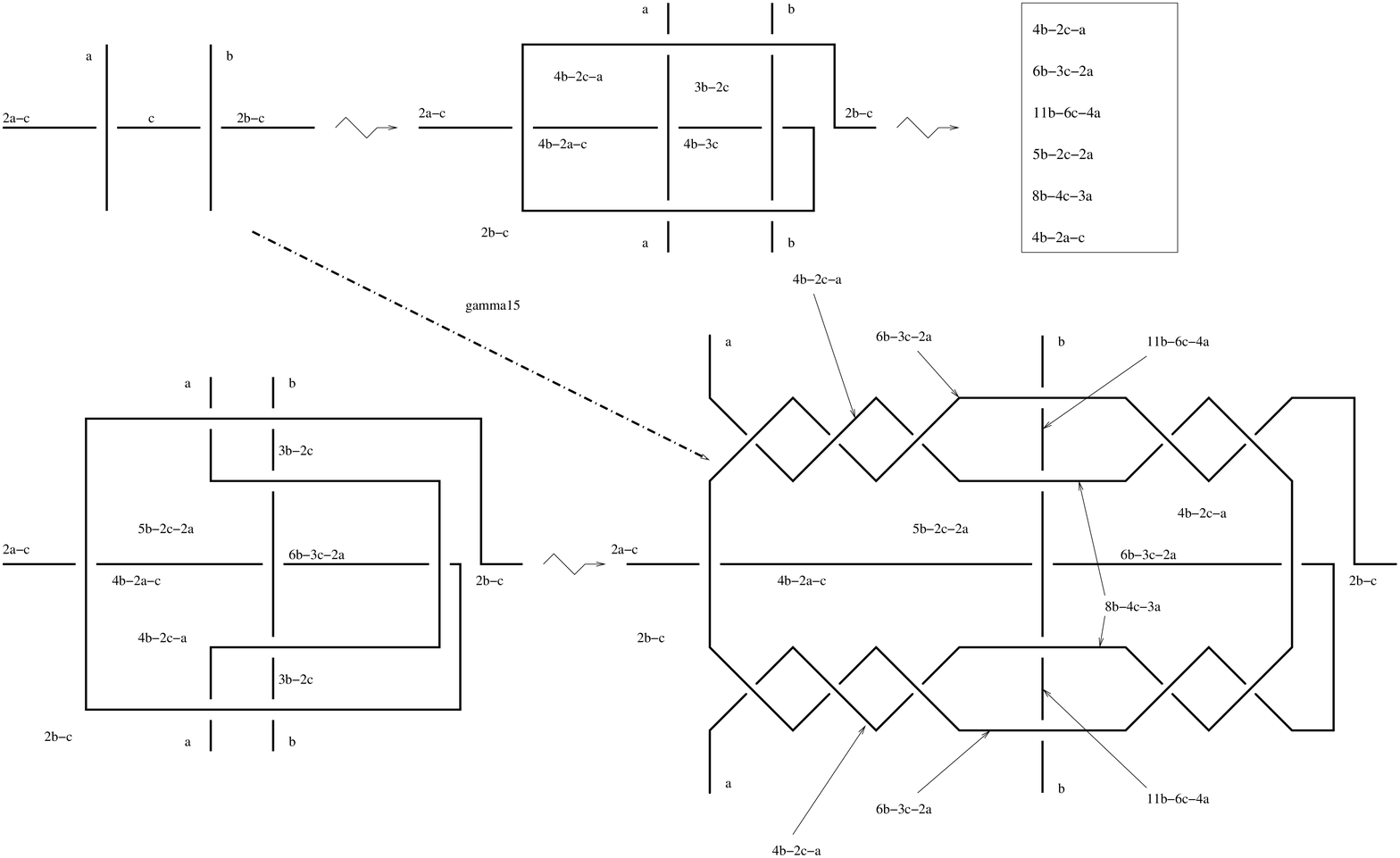}}}
	\caption{Transformation $\gamma_9$.}\label{fig:diffunder15}
\end{figure}



\begin{figure}[!ht]
	\psfrag{gamma19}{\huge$\gamma_{10}$}
	\psfrag{gamma6}{\huge$\gamma_6$}
	\psfrag{a}{\huge$a$}
	\psfrag{b}{\huge$b$}
	\psfrag{c}{\huge$c$}
	\psfrag{2b-c}{\huge$2b-c$}
	\psfrag{2a-c}{\huge$2a-c$}
	\psfrag{2b-2a+c}{\huge$2b-2a+c$}
	\psfrag{2b-4a+3c}{\huge$2b-4a+3c$}
	\psfrag{3a-2c}{\huge$3a-2c$}
	\psfrag{4b-6a+3c}{\huge$4b-6a+3c$}
	\psfrag{4b-8a+5c}{\huge$4b-8a+5c$}
	\psfrag{4b-7a+4c}{\huge$4b-7a+4c$}
	\psfrag{3b-4a+2c}{\huge$3b-4a+2c$}
	\psfrag{4a-3c}{\huge$4a-3c$} \centerline{\scalebox{.4}{\includegraphics{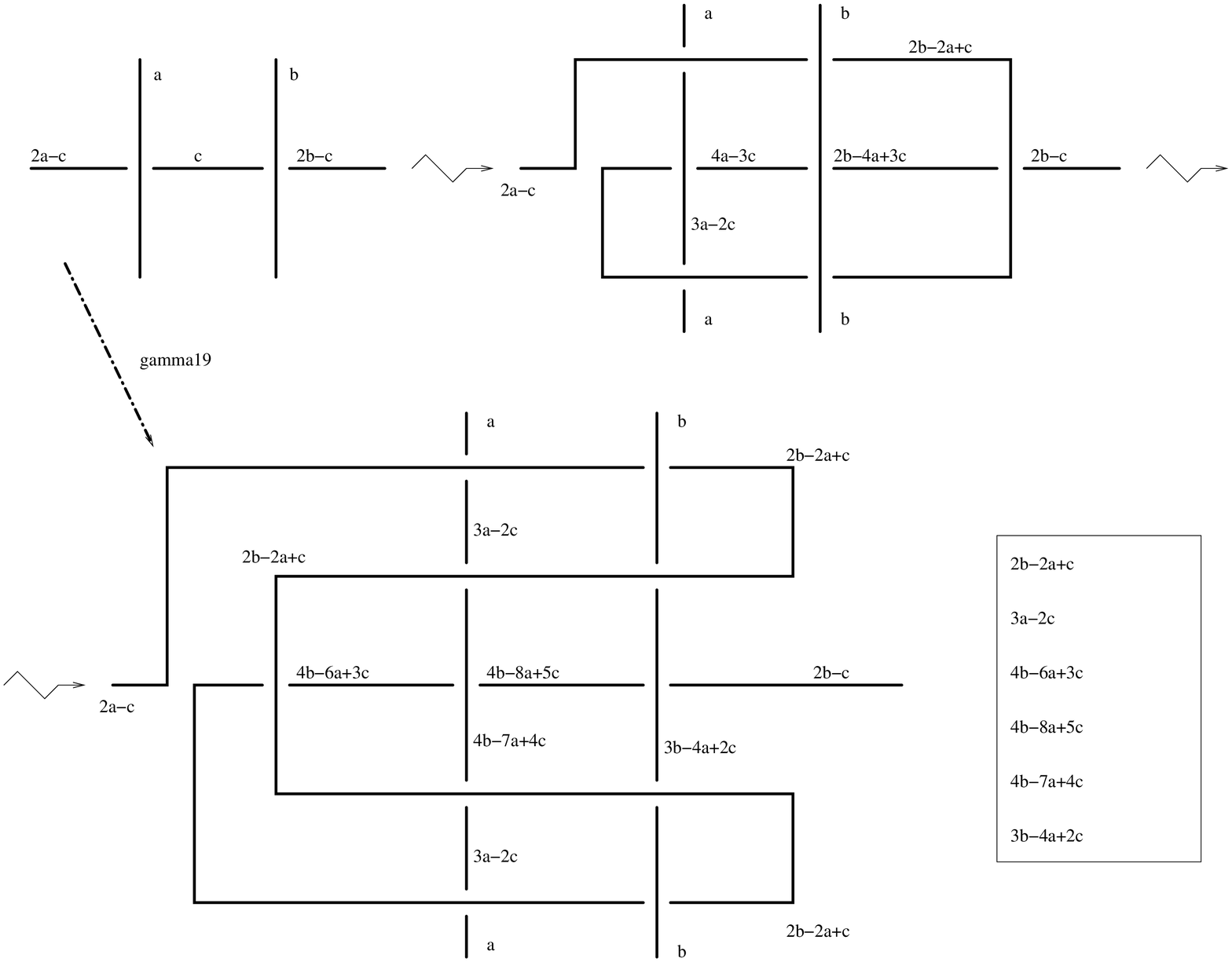}}}
	\caption{Transformation $\gamma_{10}$.}\label{fig:diffunder19}
\end{figure}



\begin{figure}[!ht]
	\psfrag{a}{\huge$a$}
	\psfrag{b}{\huge$b$}
	\psfrag{c}{\huge$c$}
	\psfrag{2b-c}{\huge$2b-c$}
	\psfrag{2b-a}{\huge$2b-a$}
	\psfrag{2b-2a+c}{\huge$2b-2a+c$}
	\psfrag{2b-4a+3c}{\huge$2b-4a+3c$}
	\psfrag{2b-3a+2c}{\huge$2b-3a+2c$}
	\psfrag{4b-5a+2c}{\huge$4b-5a+2c$}
	\psfrag{3b-4a+2c}{\huge$3b-4a+2c$}
	\psfrag{2a-c}{\huge$2a-c$} \centerline{\scalebox{.4}{\includegraphics{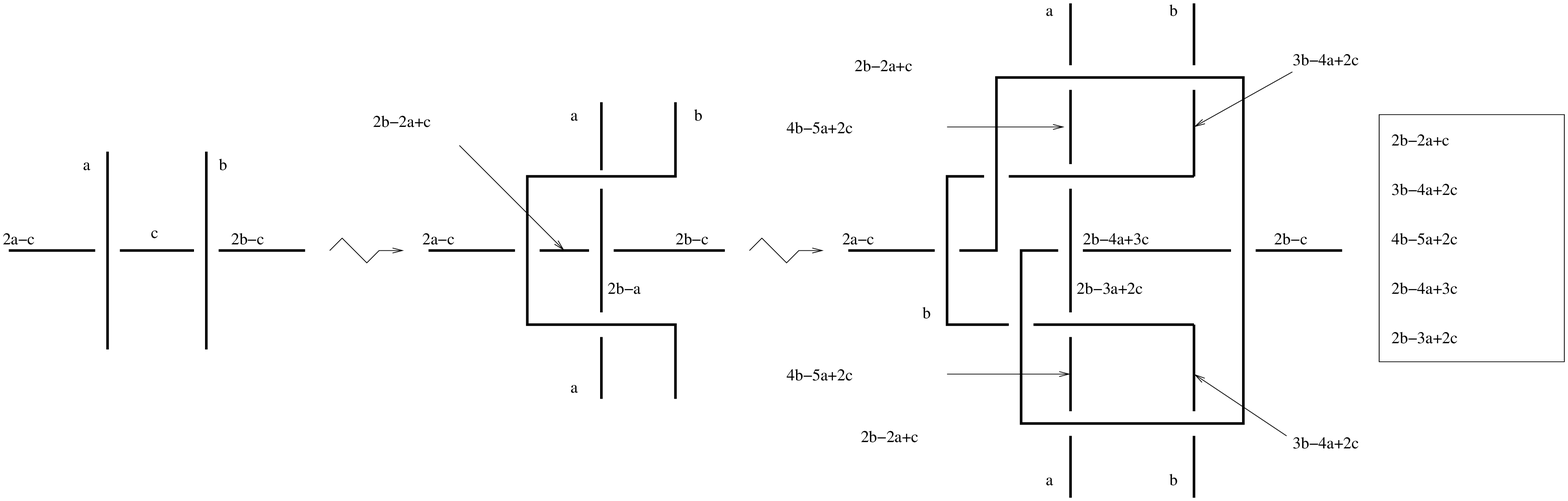}}}
	\caption{Transformation $\gamma_{11}$.}\label{fig:diffunder24}
\end{figure}



\begin{figure}[!ht]
	\psfrag{a}{\huge$a$}
	\psfrag{b}{\huge$b$}
	\psfrag{c}{\huge$c$}
	\psfrag{2b-c}{\huge$2b-c$}
	\psfrag{2b-a}{\huge$2b-a$}
	\psfrag{3c-2a}{\huge$3c-2a$}
	\psfrag{2b-2a+c}{\huge$2b-2a+c$}
	\psfrag{2b-4a+3c}{\huge$2b-4a+3c$}
	\psfrag{4b-5a+2c}{\huge$4b-5a+2c$}
	\psfrag{4b-5c+2a}{\huge$4b-5c+2a$}
	\psfrag{3b-4a+2c}{\huge$3b-4a+2c$}
	\psfrag{2a-c}{\huge$2a-c$}
	\psfrag{2c-a}{\huge$2c-a$}
	\psfrag{3a-2c}{\huge$3a-2c$}
	\psfrag{4c-3a}{\huge$4c-3a$}
	\psfrag{4a-3c}{\huge$4a-3c$}
	\psfrag{5c-4a}{\huge$5c-4a$}
	\psfrag{6c-5a}{\huge$6c-5a$}
	\psfrag{17c-10a-6b}{\huge$17c-10a-6b$}
	\psfrag{13c-8a-4b}{\huge$13c-8a-4b$}
	\psfrag{8c-5a-2b}{\huge$8c-5a-2b$}
	\psfrag{9c-6a-2b}{\huge$9c-6a-2b$}
	\psfrag{6c-4a-b}{\huge$6c-4a-b$}
	\psfrag{7c-4a-2b}{\huge$7c-4a-2b$}
	\psfrag{7c-6a}{\huge$7c-6a$}
	 \centerline{\scalebox{.4}{\includegraphics{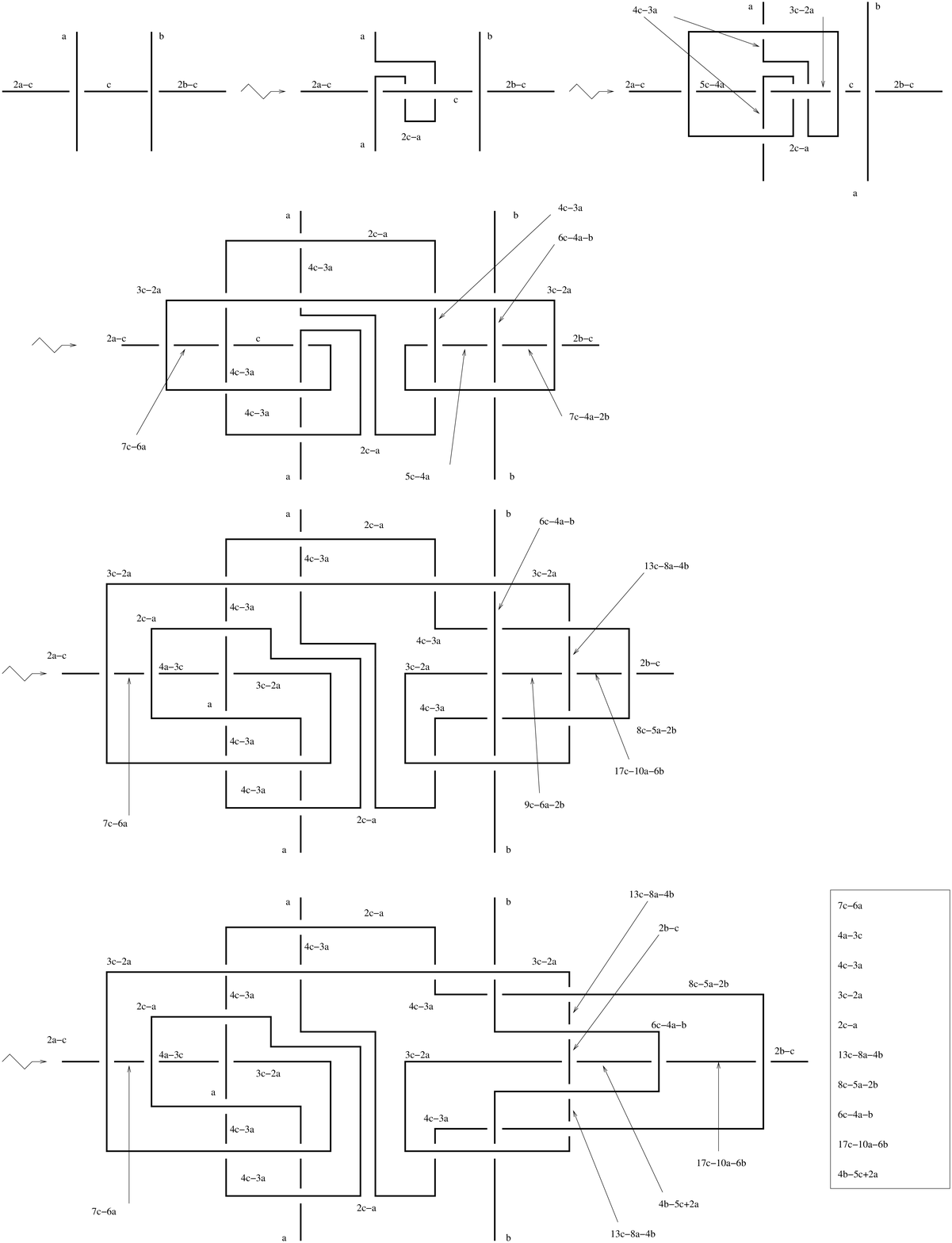}}}
	\caption{Transformation $\gamma_{12}$.}\label{fig:diffunder26}
\end{figure}

\begin{figure}[!ht]
	\psfrag{gamma19}{\huge$\gamma_{10}$}
	\psfrag{gamma6}{\huge$\gamma_6$}
	\psfrag{a}{\huge$a$}
	\psfrag{b}{\huge$b$}
	\psfrag{c}{\huge$c$}
	\psfrag{2b-c}{\huge$2b-c$}
	\psfrag{2a-c}{\huge$2a-c$}
	\psfrag{2b-2a+c}{\huge$2b-2a+c$}
	\psfrag{2b-4a+3c}{\huge$2b-4a+3c$}
	\psfrag{3a-2c}{\huge$3a-2c$}
	\psfrag{4b-6a+3c}{\huge$4b-6a+3c$}
	\psfrag{4b-8a+5c}{\huge$4b-8a+5c$}
	\psfrag{4b-7a+4c}{\huge$4b-7a+4c$}
	\psfrag{3b-4a+2c}{\huge$3b-4a+2c$}
	\psfrag{6b-12a+7c}{\huge$6b-12a+7c$}
	\psfrag{4a-3c}{\huge$4a-3c$}
\centerline{\scalebox{.4}{\includegraphics{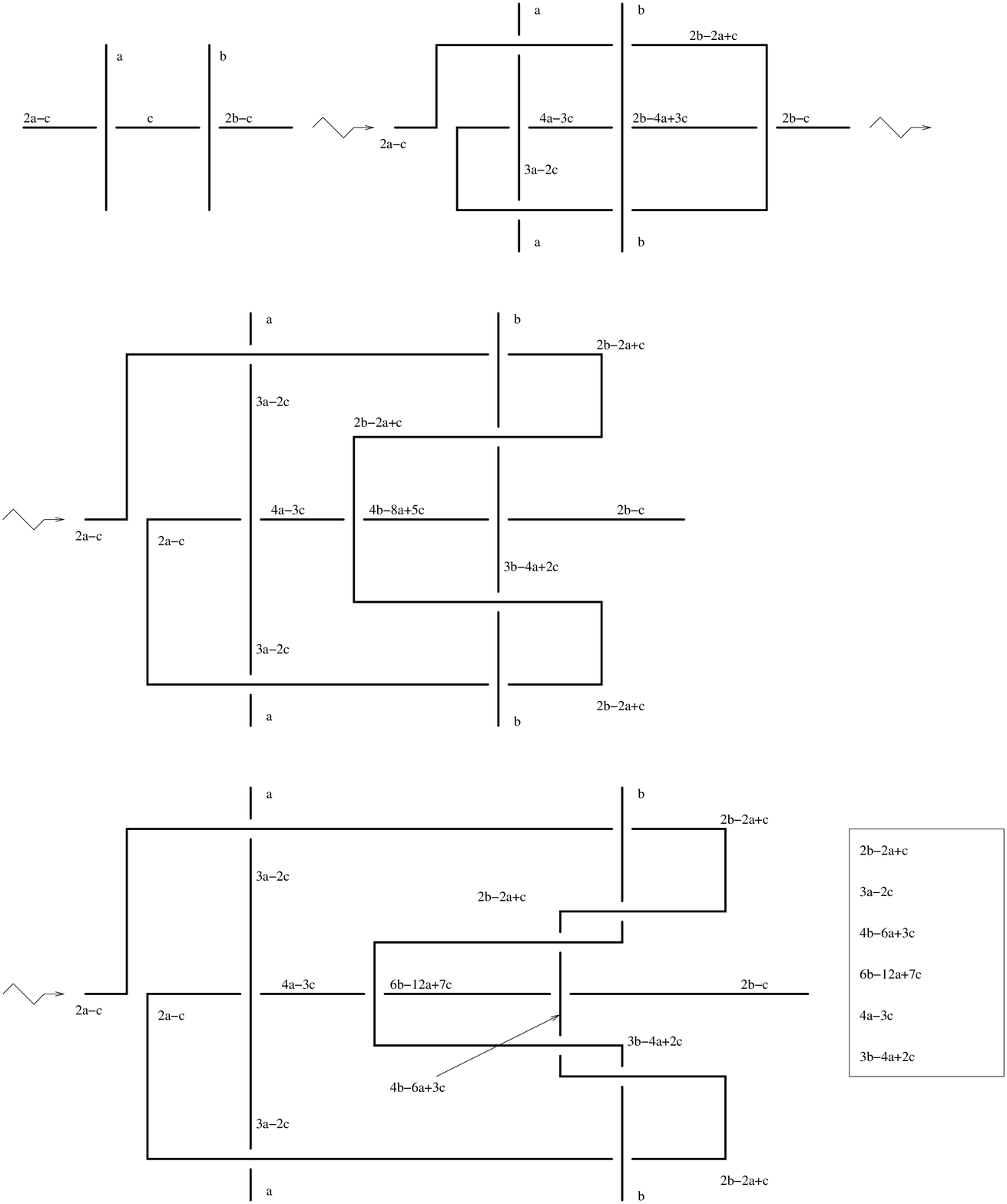}}}
	\caption{Transformation $\gamma_{13}$.}\label{fig:diffunder75}
\end{figure}

\begin{figure}[!ht]
	\psfrag{gamma19}{\huge$\gamma_{10}$}
	\psfrag{gamma6}{\huge$\gamma_6$}
	\psfrag{a}{\huge$a$}
	\psfrag{b}{\huge$b$}
	\psfrag{c}{\huge$c$}
	\psfrag{2b-c}{\huge$2b-c$}
	\psfrag{2a-c}{\huge$2a-c$}
	\psfrag{2b-2a+c}{\huge$2b-2a+c$}
	\psfrag{2b-4a+3c}{\huge$2b-4a+3c$}
	\psfrag{3a-2c}{\huge$3a-2c$}
	\psfrag{6a-b-4c}{\huge$6a-b-4c$}
	\psfrag{8a-2b-5c}{\huge$8a-2b-5c$}
	\psfrag{4b-6a+3c}{\huge$4b-6a+3c$}
	\psfrag{4b-8a+5c}{\huge$4b-8a+5c$}
	\psfrag{4b-7a+4c}{\huge$4b-7a+4c$}
	\psfrag{3b-4a+2c}{\huge$3b-4a+2c$}
	\psfrag{6b-12a+7c}{\huge$6b-12a+7c$}
	\psfrag{10a-7c-2b}{\huge$10a-7c-2b$}
	\psfrag{10a-3b-6c}{\huge$10a-3b-6c$}
	\psfrag{6a-3c-2b}{\huge$6a-3c-2b$}
	\psfrag{14a-4b-9c}{\huge$14a-4b-9c$}
	\psfrag{17a-6b-10c}{\huge$17a-6b-10c$}
	\psfrag{20a-8b-11c}{\huge$20a-8b-11c$}
	\psfrag{4a-3c}{\huge$4a-3c$}
\centerline{\scalebox{.4}{\includegraphics{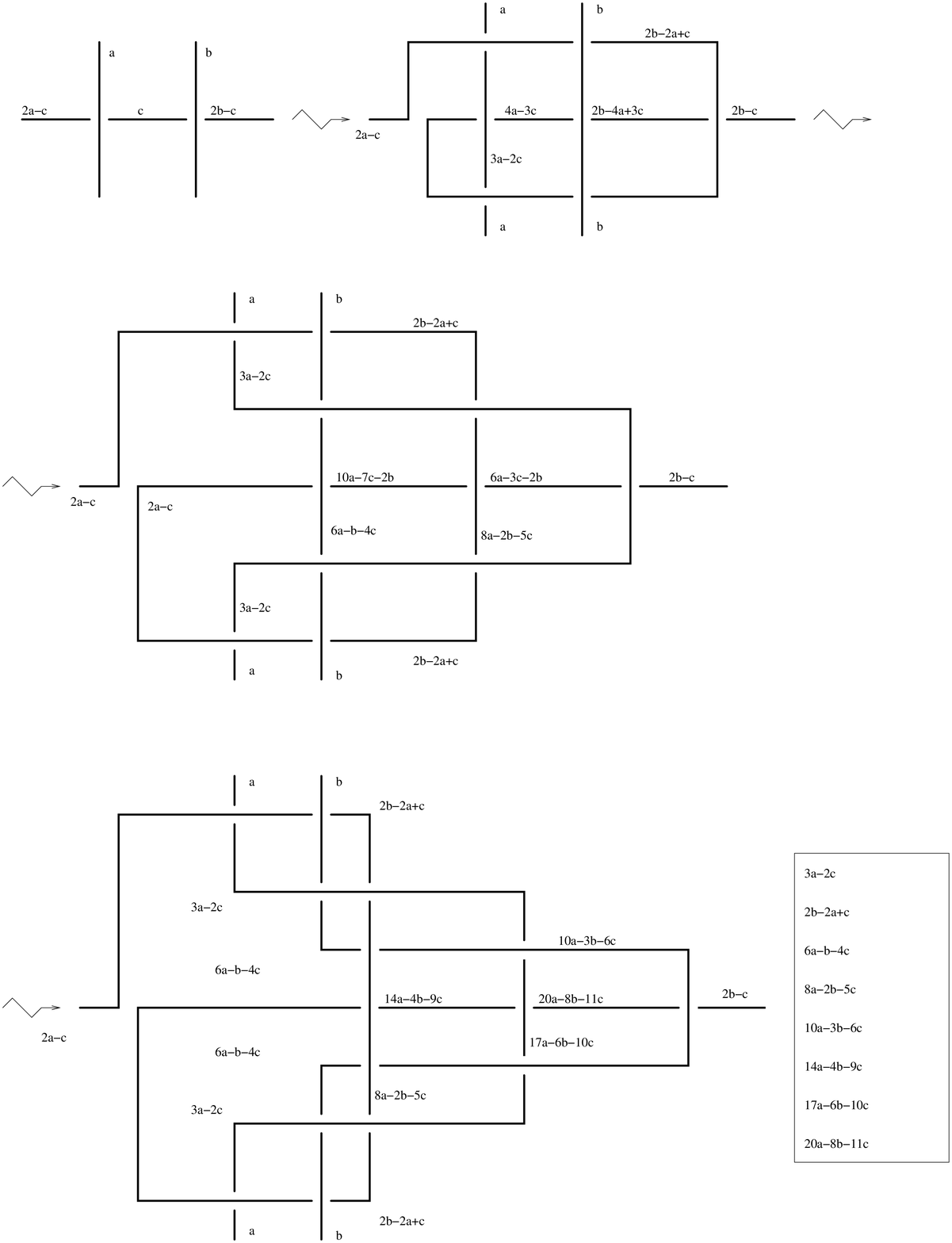}}}
	\caption{Transformation $\gamma_{14}$.}\label{fig:diffunder76}
\end{figure}









\begin{figure}[!ht]
	\psfrag{delta1}{\huge$\delta_{1}$}
	\psfrag{delta2}{\huge$\delta_{2}$}
	\psfrag{a}{\huge$a$}
	\psfrag{c}{\huge$c$}
	\psfrag{19c-18a}{\huge$19c-18a$}
	\psfrag{18c-17a}{\huge$18c-17a$}
	\psfrag{23c-22a}{\huge$23c-22a$}
	\psfrag{13c-12a}{\huge$13c-12a$}
	\psfrag{20c-19a}{\huge$20c-19a$}
	\psfrag{2a-c}{\huge$2a-c$}
	\psfrag{3a-2c}{\huge$3a-2c$}
	\psfrag{6a-5c}{\huge$6a-5c$}
	\psfrag{5c-4a}{\huge$5c-4a$}
	\psfrag{6c-5a}{\huge$6c-5a$}
	\psfrag{9a-8c}{\huge$9a-8c$}
	\psfrag{3c-2a}{\huge$3c-2a$}
	\psfrag{8c-7a}{\huge$8c-7a$}
	\psfrag{7c-6a}{\huge$7c-6a$}
	 \centerline{\scalebox{.4}{\includegraphics{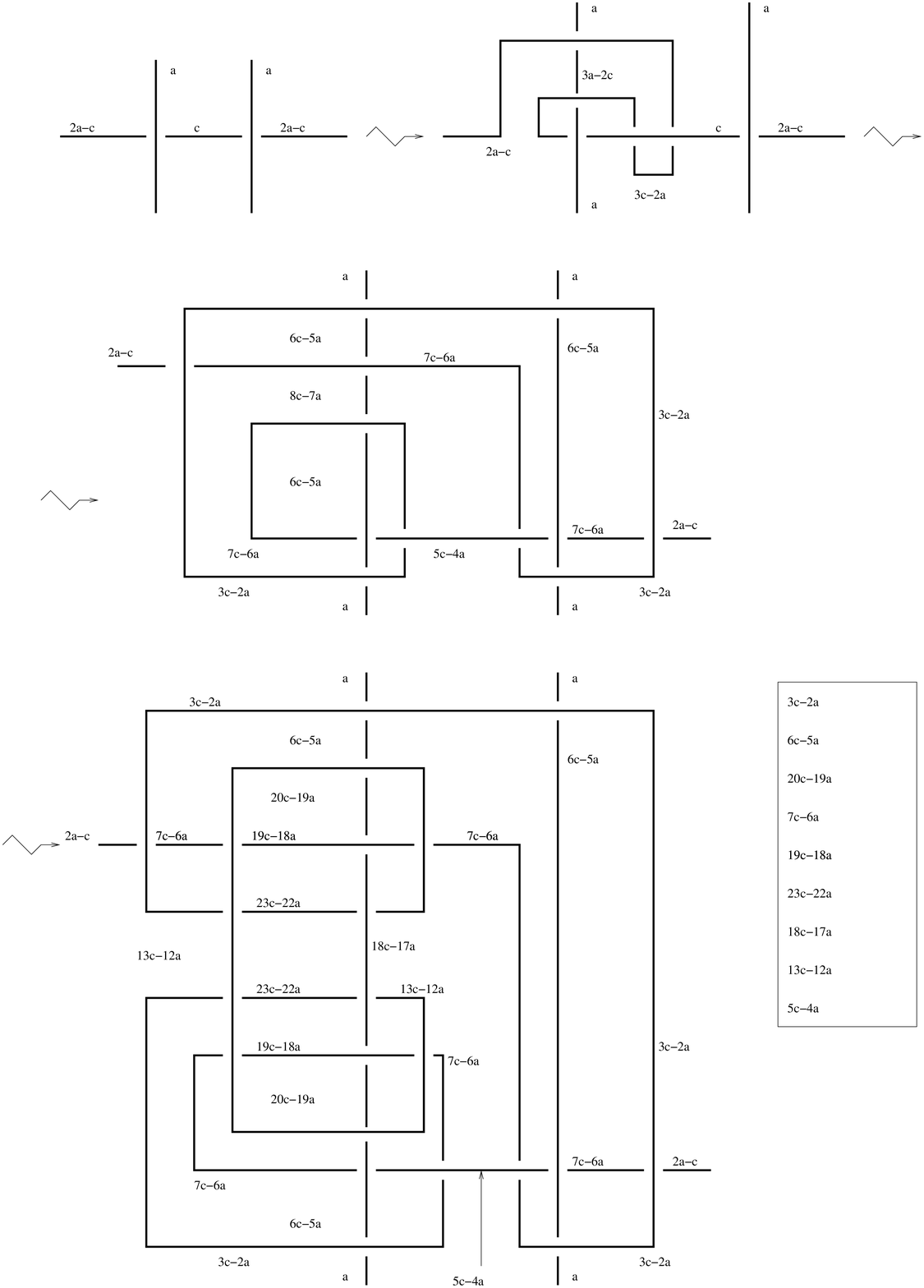}}}
	\caption{Transformation $\delta_4$.}\label{fig:sameunder10}
\end{figure}





\begin{figure}[!ht]
	\psfrag{delta1}{\huge$\delta_{1}$}
	\psfrag{delta3}{\huge$\delta_{3}$}
	\psfrag{a}{\huge$a$}
	\psfrag{c}{\huge$c$}
	\psfrag{12a-11c}{\huge$12a-11c$}
	\psfrag{19a-18c}{\huge$19a-18c$}
	\psfrag{9a-8c}{\huge$9a-8c$}
	\psfrag{2a-c}{\huge$2a-c$}
	\psfrag{4a-3c}{\huge$4a-3c$}
	\psfrag{2c-a}{\huge$2c-a$}
	\psfrag{3a-2c}{\huge$3a-2c$}
	\psfrag{6a-5c}{\huge$6a-5c$}
	\psfrag{7a-6c}{\huge$7a-6c$}
	\psfrag{5a-4c}{\huge$5a-4c$}
	\psfrag{5c-4a}{\huge$5c-4a$}
	 \centerline{\scalebox{.4}{\includegraphics{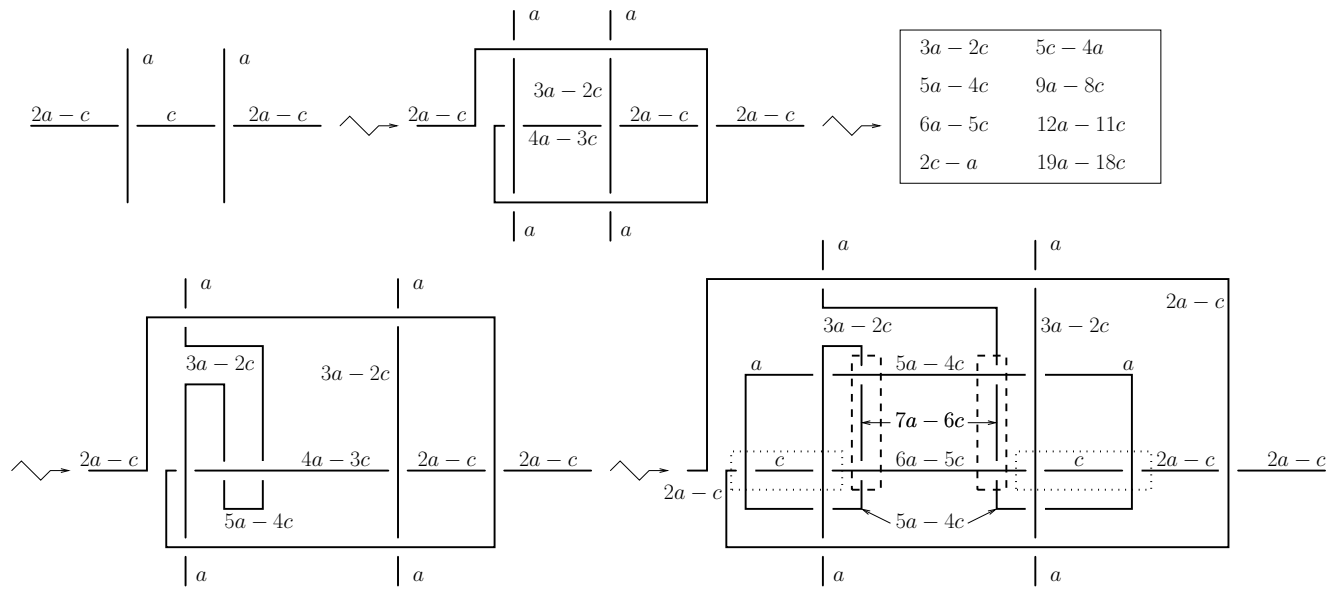}}}
	\caption{Transformation $\delta_{5}$. The dotted boxes and broken line boxes will be addressed in Figures \ref{fig:sameunder13dottedbox} and \ref{fig:sameunder13brokenlinebox}, respectively.}\label{fig:sameunder13}
\end{figure}

\clearpage

\begin{figure}[!ht]
	\psfrag{delta1}{\huge$\delta_{1}$}
	\psfrag{delta3}{\huge$\delta_{3}$}
	\psfrag{a}{\huge$a$}
	\psfrag{2a-c}{\huge$2a-c$}
	\psfrag{c}{\huge$c$}
	\psfrag{2c-a}{\huge$2c-a$}
	\psfrag{3a-2c}{\huge$3a-2c$}
	\psfrag{6a-5c}{\huge$6a-5c$}
	\psfrag{5c-4a}{\huge$5c-4a$}
	\psfrag{11}{\huge$\mathbf{11}$}
	 \centerline{\scalebox{.4}{\includegraphics{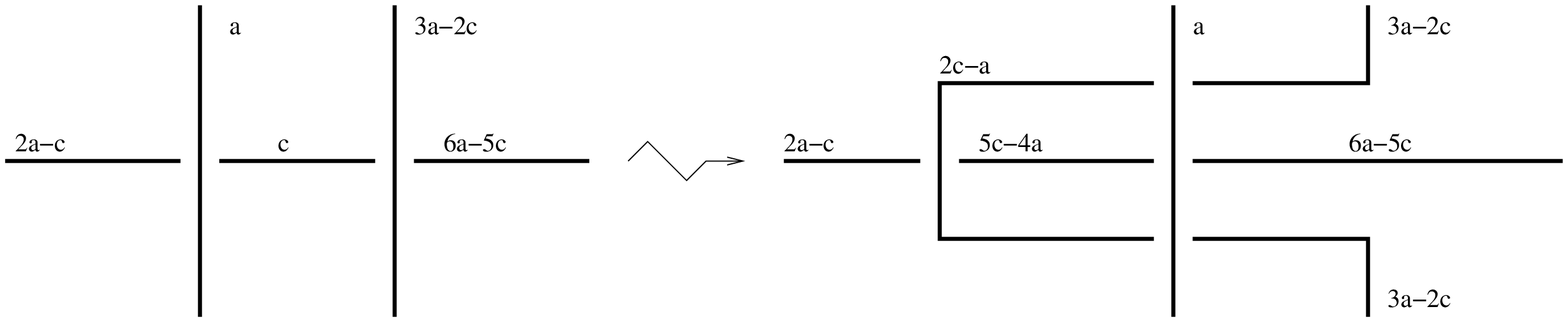}}}
	\caption{Transformation $\delta_{5}$: addressing the dotted box from Figure \ref{fig:sameunder13}.}\label{fig:sameunder13dottedbox}
\end{figure}





\begin{figure}[!ht]
	\psfrag{delta1}{\huge$\delta_{1}$}
	\psfrag{delta3}{\huge$\delta_{3}$}
	\psfrag{a}{\huge$a$}
	\psfrag{7a-6c}{\huge$7a-6c$}
	\psfrag{2c-a}{\huge$2c-a$}
	\psfrag{3a-2c}{\huge$3a-2c$}
	\psfrag{5a-4c}{\huge$5a-4c$}
	\psfrag{6a-5c}{\huge$6a-5c$}
	\psfrag{9a-8c}{\huge$9a-8c$}
	\psfrag{13a-12c}{\huge$13a-12c$}
	\psfrag{19a-18c}{\huge$19a-18c$}
	\psfrag{12a-11c}{\huge$12a-11c$}
	 \centerline{\scalebox{.4}{\includegraphics{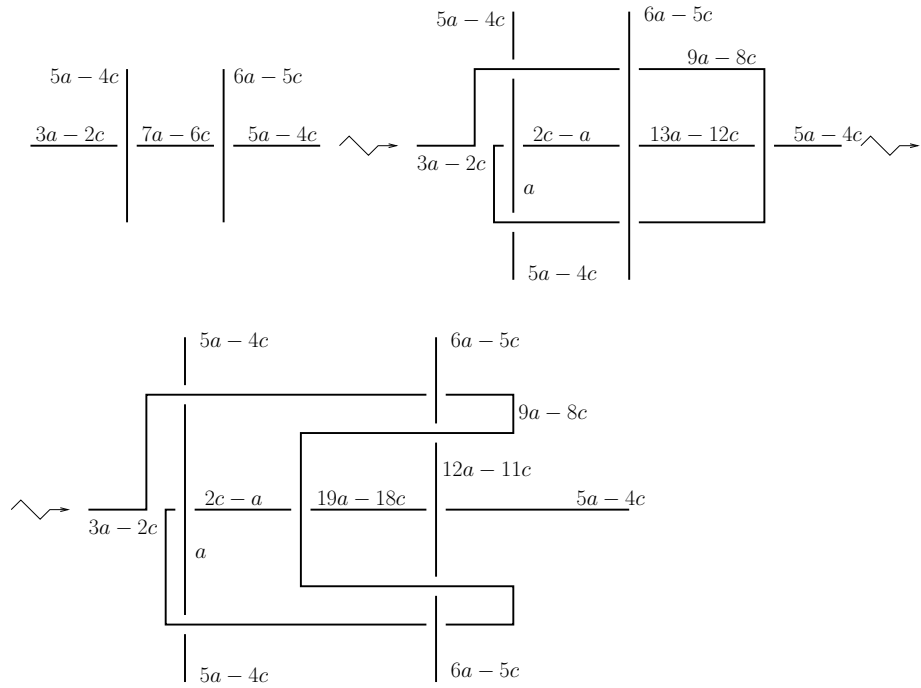}}}
	\caption{Transformation $\delta_{5}$: addressing the broken line box  from Figure \ref{fig:sameunder13}.}\label{fig:sameunder13brokenlinebox}
\end{figure}



\begin{figure}[!ht]
	\psfrag{delta1}{\huge$\delta_{1}$}
	\psfrag{delta3}{\huge$\delta_{3}$}
	\psfrag{a}{\huge$a$}
	\psfrag{c}{\huge$c$}
	\psfrag{8a-7c}{\huge$8a-7c$}
	\psfrag{7a-6c}{\huge$7a-6c$}
	\psfrag{2c-a}{\huge$2c-a$}
	\psfrag{2a-c}{\huge$2a-c$}
	\psfrag{3a-2c}{\huge$3a-2c$}
	\psfrag{4c-3a}{\huge$4c-3a$}
	\psfrag{4a-3c}{\huge$4a-3c$}
	\psfrag{5a-4c}{\huge$5a-4c$}
	\psfrag{6a-5c}{\huge$6a-5c$}
	\psfrag{9a-8c}{\huge$9a-8c$}
	\psfrag{13a-12c}{\huge$13a-12c$}
	\psfrag{19a-18c}{\huge$19a-18c$}
	\psfrag{12a-11c}{\huge$12a-11c$}
	\psfrag{11a-10c}{\huge$11a-10c$}
	 \centerline{\scalebox{.4}{\includegraphics{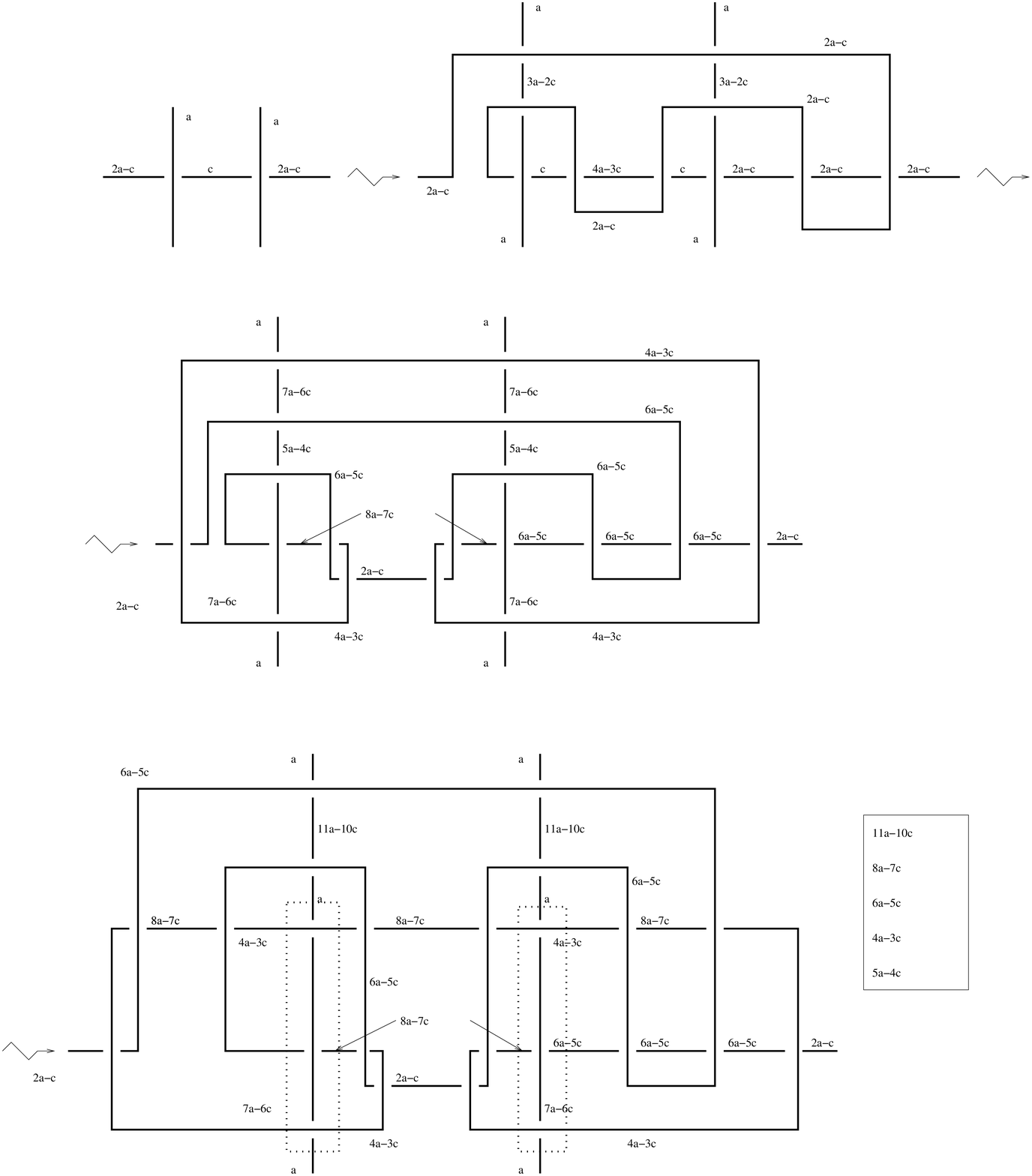}}}
	\caption{Transformation $\delta_{6}$. The dotted boxes will be addressed in Figure \ref{fig:sameunder55dottedbox}.}\label{fig:sameunder55}
\end{figure}



\begin{figure}[!ht]
	\psfrag{delta1}{\huge$\delta_{1}$}
	\psfrag{delta3}{\huge$\delta_{3}$}
	\psfrag{a}{\huge$a$}
	\psfrag{c}{\huge$c$}
	\psfrag{8a-7c}{\huge$8a-7c$}
	\psfrag{7a-6c}{\huge$7a-6c$}
	\psfrag{2c-a}{\huge$2c-a$}
	\psfrag{2a-c}{\huge$2a-c$}
	\psfrag{3a-2c}{\huge$3a-2c$}
	\psfrag{4c-3a}{\huge$4c-3a$}
	\psfrag{4a-3c}{\huge$4a-3c$}
	\psfrag{5a-4c}{\huge$5a-4c$}
	\psfrag{6a-5c}{\huge$6a-5c$}
	\psfrag{9a-8c}{\huge$9a-8c$}
	\psfrag{13a-12c}{\huge$13a-12c$}
	\psfrag{19a-18c}{\huge$19a-18c$}
	\psfrag{12a-11c}{\huge$12a-11c$}
	\psfrag{11a-10c}{\huge$11a-10c$}
	\psfrag{11}{\huge$\mathbf{11}$}
	 \centerline{\scalebox{.4}{\includegraphics{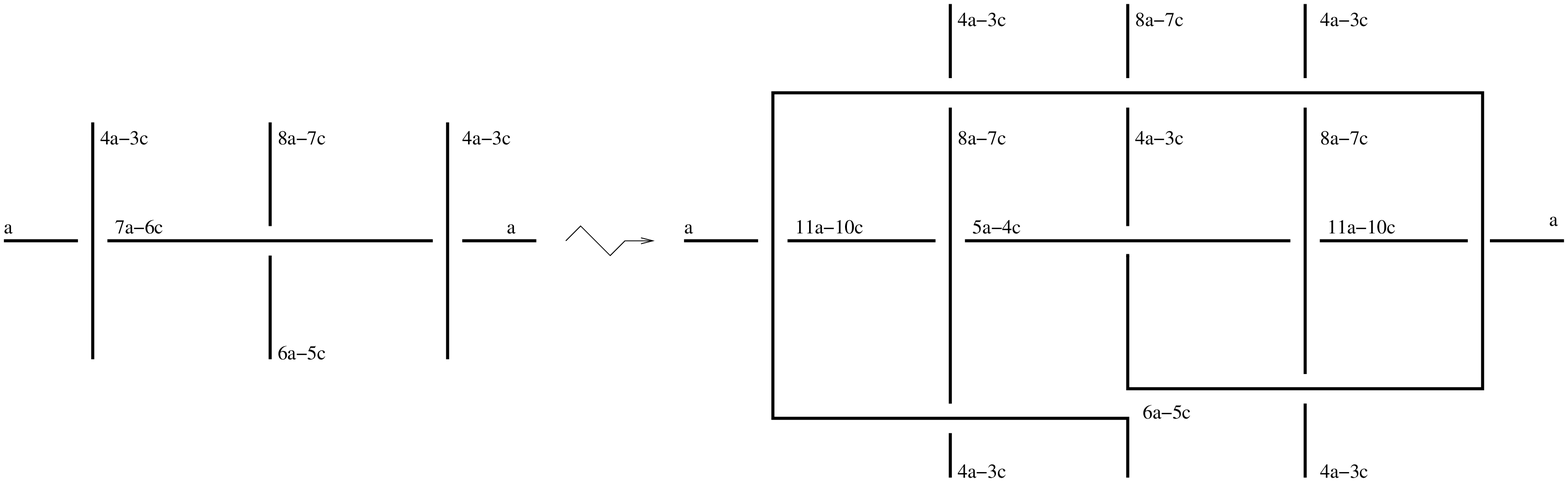}}}
	\caption{Transformation $\delta_{6}$: addressing the dotted boxes from Figure \ref{fig:sameunder55}.}\label{fig:sameunder55dottedbox}
\end{figure}

\begin{figure}[!ht]
	\psfrag{delta1}{\huge$\delta_{1}$}
	\psfrag{delta2}{\huge$\delta_{2}$}
	\psfrag{a}{\huge$a$}
	\psfrag{c}{\huge$c$}
	\psfrag{9a-8c}{\huge$9a-8c$}
	\psfrag{6a-5c}{\huge$6a-5c$}
	\psfrag{2a-c}{\huge$2a-c$}
	\psfrag{3a-2c}{\huge$3a-2c$}
	\psfrag{5c-4a}{\huge$5c-4a$}
	\psfrag{6c-5a}{\huge$6c-5a$}
	\psfrag{3c-2a}{\huge$3c-2a$}
	\psfrag{7c-6a}{\huge$7c-6a$}
	 \centerline{\scalebox{.4}{\includegraphics{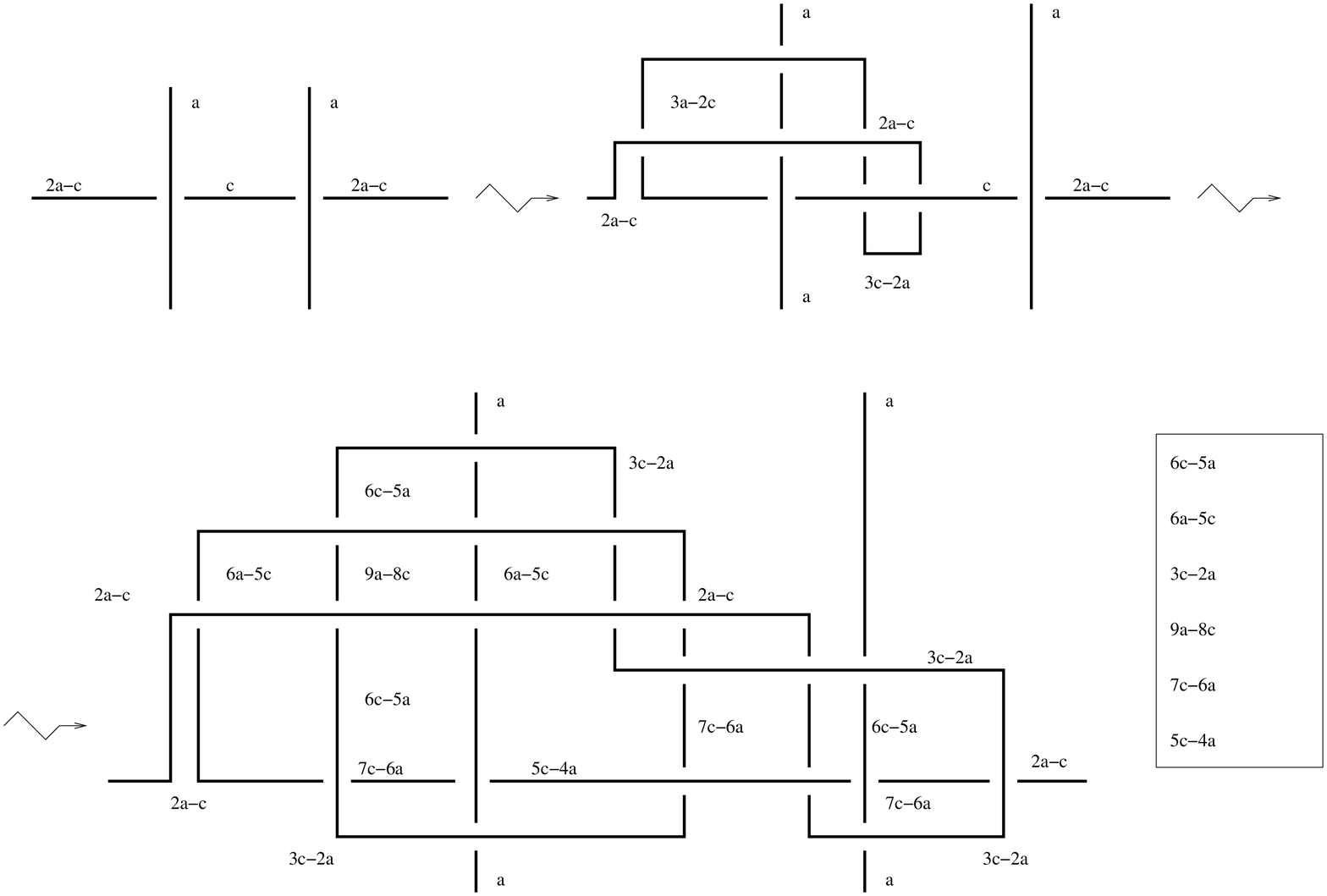}}}
	\caption{Transformation $\delta_7$.}\label{fig:sameunder8}
\end{figure}

\begin{figure}[!ht]
	\psfrag{delta1}{\huge$\delta_{1}$}
	\psfrag{delta2}{\huge$\delta_{2}$}
	\psfrag{a}{\huge$a$}
	\psfrag{c}{\huge$c$}
	\psfrag{12a-11c}{\huge$12a-11c$}
	\psfrag{2a-c}{\huge$2a-c$}
	\psfrag{2c-a}{\huge$2c-a$}
	\psfrag{3a-2c}{\huge$3a-2c$}
	\psfrag{4a-3c}{\huge$4a-3c$}
	\psfrag{4c-3a}{\huge$4c-3a$}
	\psfrag{6a-5c}{\huge$6a-5c$}
	\psfrag{5a-4c}{\huge$5a-4c$}
	\psfrag{5c-4a}{\huge$5c-4a$}
	\psfrag{11c-10a}{\huge$11c-10a$}
	\psfrag{9a-8c}{\huge$9a-8c$}
	\psfrag{3c-2a}{\huge$3c-2a$}
	 \centerline{\scalebox{.4}{\includegraphics{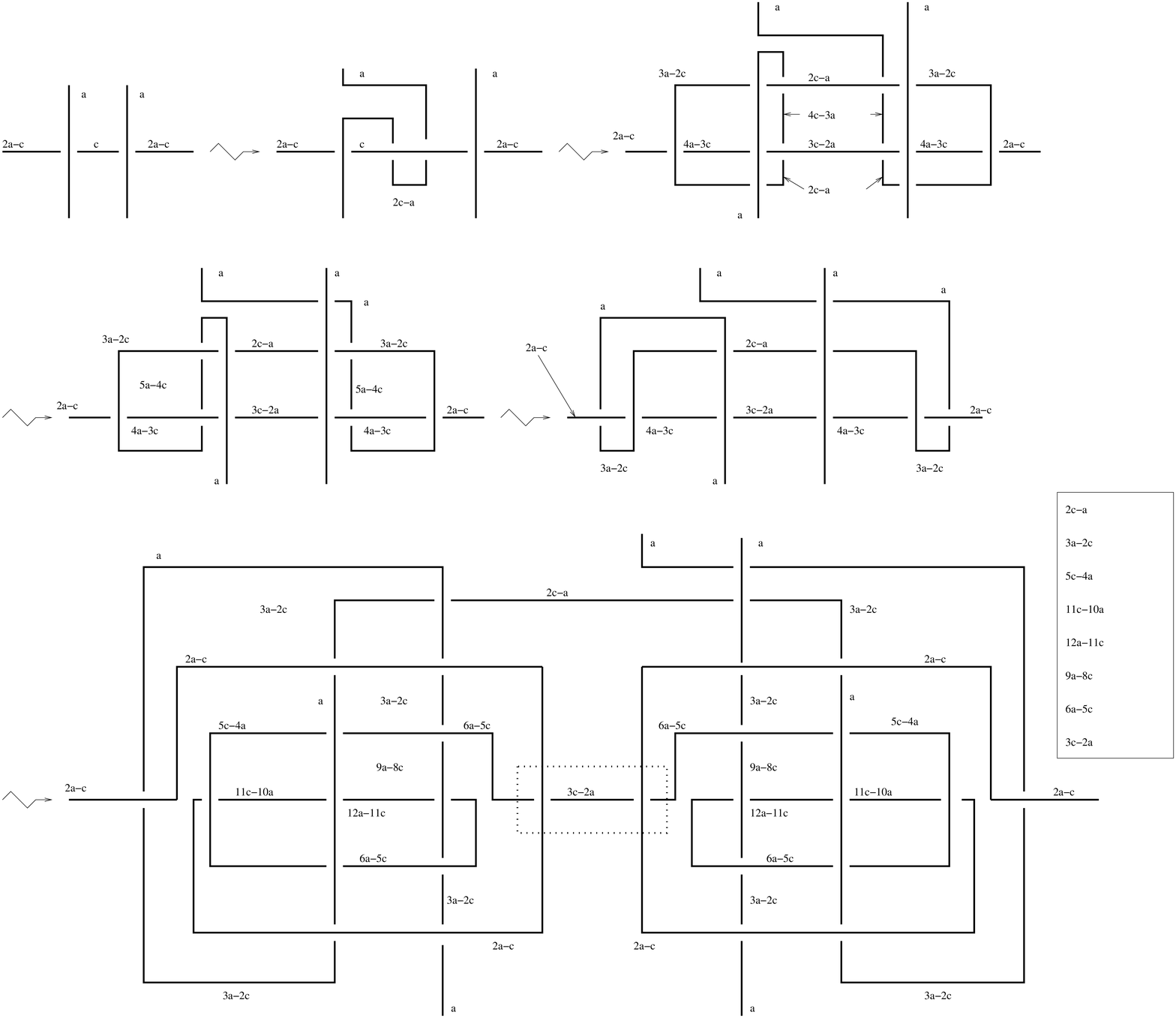}}}
	\caption{Transformation $\delta_8$.}\label{fig:sameunder18}
\end{figure}

\begin{figure}[!ht]
	\psfrag{delta1}{\huge$\delta_{1}$}
	\psfrag{delta2}{\huge$\delta_{2}$}
	\psfrag{a}{\huge$a$}
	\psfrag{c}{\huge$c$}
	\psfrag{12a-11c}{\huge$12a-11c$}
	\psfrag{2a-c}{\huge$2a-c$}
	\psfrag{2c-a}{\huge$2c-a$}
	\psfrag{3a-2c}{\huge$3a-2c$}
	\psfrag{4a-3c}{\huge$4a-3c$}
	\psfrag{6a-5c}{\huge$6a-5c$}
	\psfrag{5a-4c}{\huge$5a-4c$}
	\psfrag{5c-4a}{\huge$5c-4a$}
	\psfrag{11c-10a}{\huge$11c-10a$}
	\psfrag{9a-8c}{\huge$9a-8c$}
	\psfrag{10a-9c}{\huge$10a-9c$}
	\psfrag{14a-13c}{\huge$14a-13c$}
	\psfrag{3c-2a}{\huge$3c-2a$}
	 \centerline{\scalebox{.4}{\includegraphics{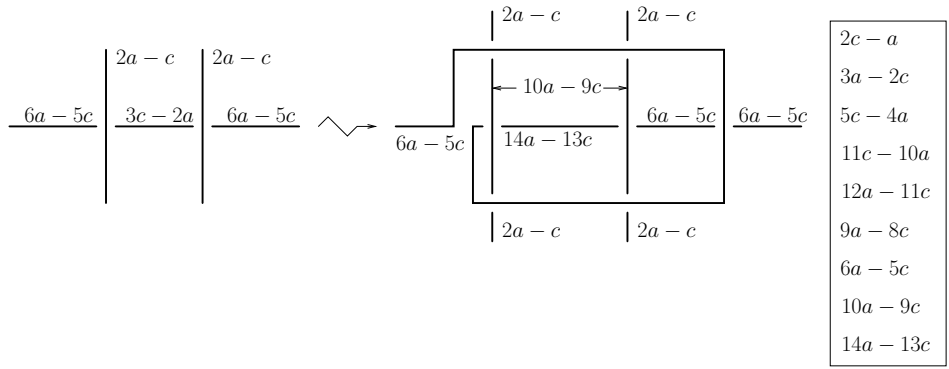}}}
	\caption{Transformation $\delta_9$: it is obtained from transformation $\delta_8$ treating its dotted box as in the current figure. The list presented in this figure is the complete list corresponding to the full transformation $\delta_9$.}\label{fig:sameunder18plus}
\end{figure}



\begin{figure}[!ht]
	\psfrag{delta1}{\huge$\delta_{1}$}
	\psfrag{delta2}{\huge$\delta_{2}$}
	\psfrag{a}{\huge$a$}
	\psfrag{c}{\huge$c$}
	\psfrag{2a-c}{\huge$2a-c$}
	\psfrag{2c-a}{\huge$2c-a$}
	\psfrag{4c-3a}{\huge$4c-3a$}
	\psfrag{4a-3c}{\huge$4a-3c$}
	\psfrag{5c-4a}{\huge$5c-4a$}
	\psfrag{3c-2a}{\huge$3c-2a$}
	\psfrag{7c-6a}{\huge$7c-6a$}
\centerline{\scalebox{.4}{\includegraphics{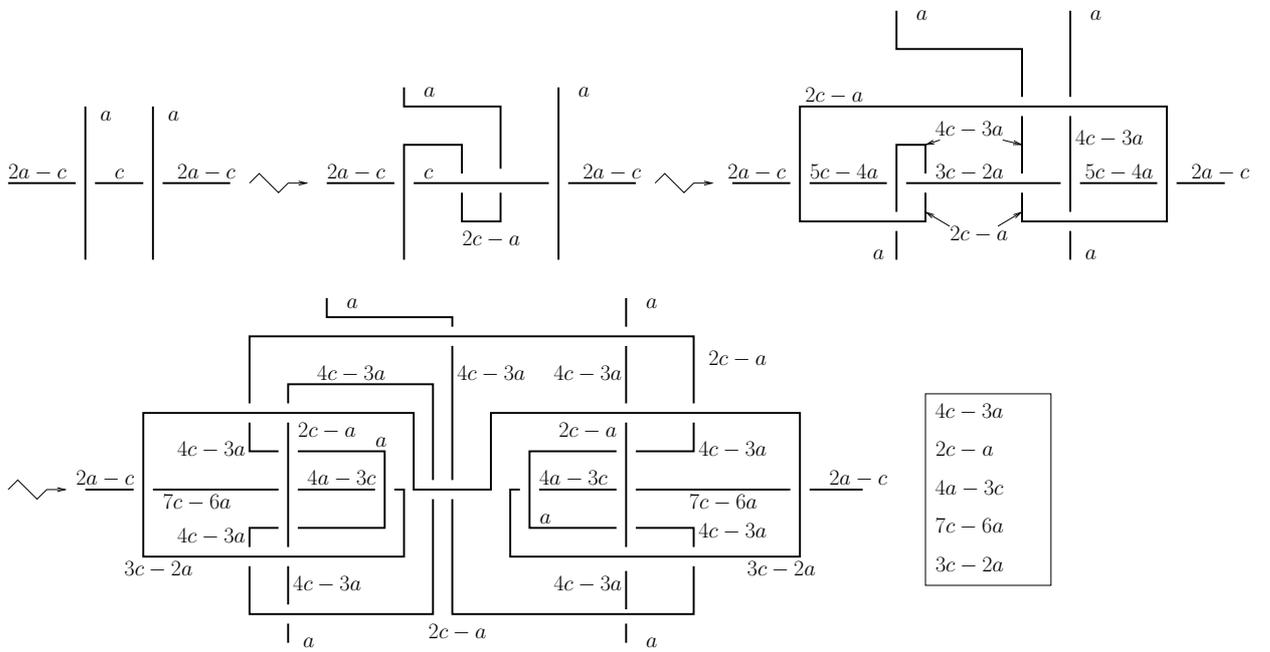}}}
	\caption{Transformation $\delta_{10}$.}\label{fig:sameunder20}
\end{figure}

\clearpage

\begin{figure}[!ht]
	\psfrag{delta1}{\huge$\delta_{1}$}
	\psfrag{delta2}{\huge$\delta_{2}$}
	\psfrag{a}{\huge$a$}
	\psfrag{c}{\huge$c$}
	\psfrag{28c-27a}{\huge$28c-27a$}
	\psfrag{2a-c}{\huge$2a-c$}
	\psfrag{5a-4c}{\huge$5a-4c$}
	\psfrag{2c-a}{\huge$2c-a$}
	\psfrag{4c-3a}{\huge$4c-3a$}
	\psfrag{4a-3c}{\huge$4a-3c$}
	\psfrag{6c-5a}{\huge$6c-5a$}
	\psfrag{3c-2a}{\huge$3c-2a$}
	\psfrag{9c-8a}{\huge$9c-8a$}
	\psfrag{8c-7a}{\huge$8c-7a$}
	\psfrag{7c-6a}{\huge$7c-6a$}
	\psfrag{12c-11a}{\huge$12c-11a$}
	\psfrag{22c-21a}{\huge$22c-21a$}
	\psfrag{15c-14a}{\huge$15c-14a$}
	 \centerline{\scalebox{.4}{\includegraphics{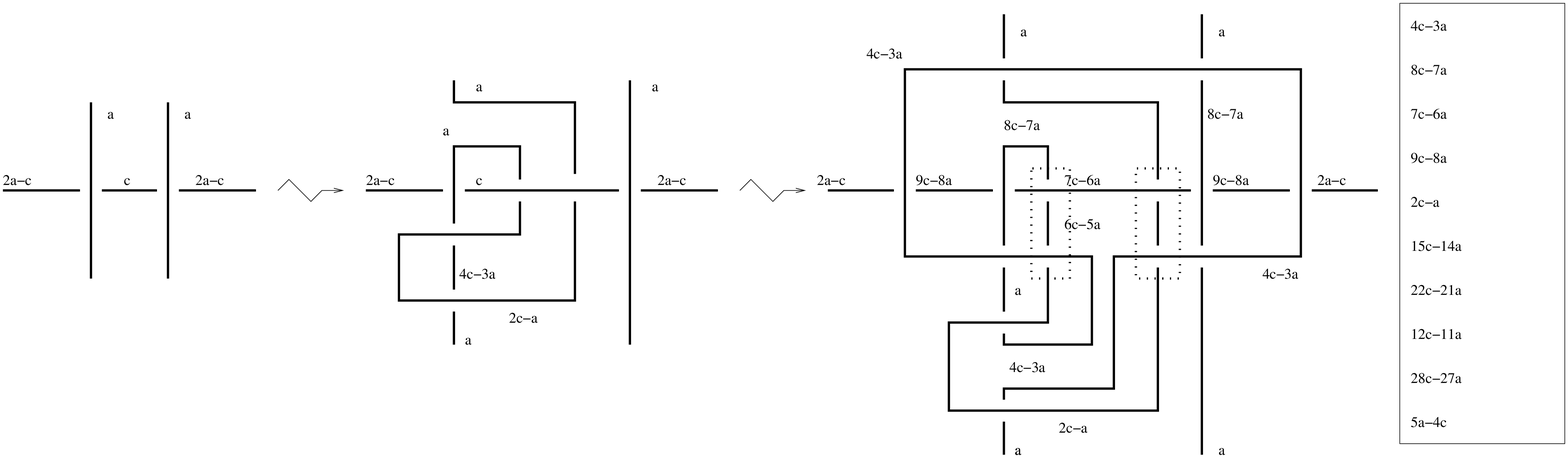}}}
	\caption{Transformation $\delta_{11}$. The dotted box will be addressed in Figure \ref{fig:sameunder22dottedbox}.}\label{fig:sameunder22}
\end{figure}

\begin{figure}[!ht]
	\psfrag{delta1}{\huge$\delta_{1}$}
	\psfrag{delta2}{\huge$\delta_{2}$}
	\psfrag{2c-a}{\huge$2c-a$}
	\psfrag{4c-3a}{\huge$4c-3a$}
	\psfrag{6c-5a}{\huge$6c-5a$}
	\psfrag{6a-5c}{\huge$6a-5c$}
	\psfrag{5a-4c}{\huge$5a-4c$}
	\psfrag{3c-2a}{\huge$3c-2a$}
	\psfrag{9c-8a}{\huge$9c-8a$}
	\psfrag{8c-7a}{\huge$8c-7a$}
	\psfrag{7c-6a}{\huge$7c-6a$}
	\psfrag{12c-11a}{\huge$12c-11a$}
	\psfrag{10c-9a}{\huge$10c-9a$}
	\psfrag{22c-21a}{\huge$22c-21a$}
	\psfrag{28c-27a}{\huge$28c-27a$}
	\psfrag{15c-14a}{\huge$15c-14a$}
	\psfrag{14c-13a}{\huge$14c-13a$}
	 \centerline{\scalebox{.4}{\includegraphics{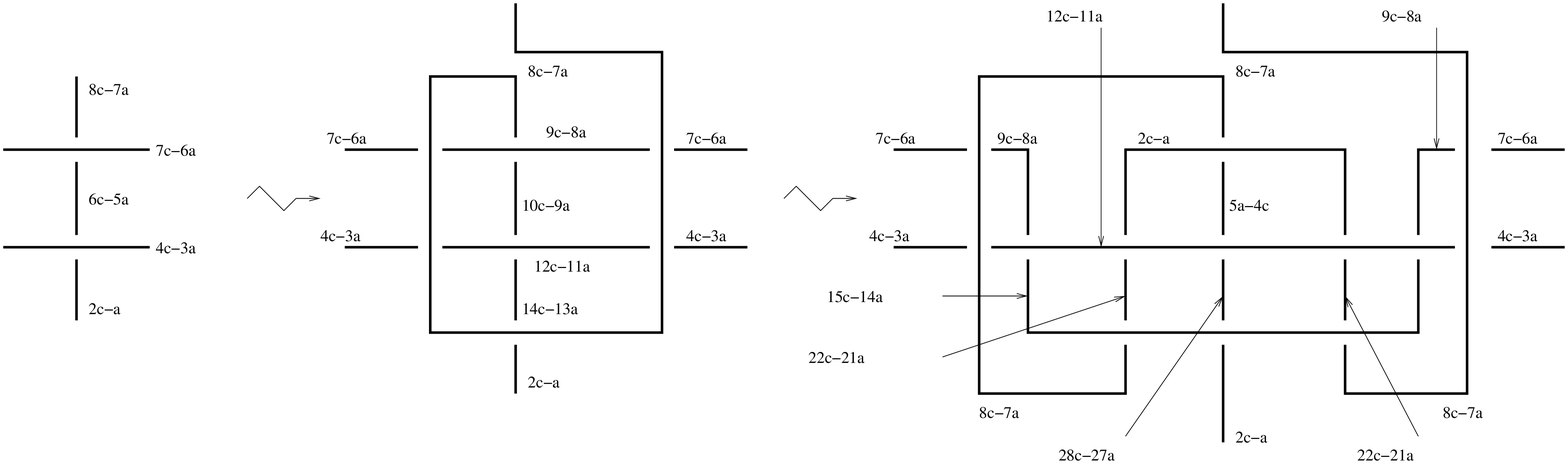}}}
	\caption{Transformation $\delta_{11}$. Addressing the dotted box from Figure \ref{fig:sameunder22}.}\label{fig:sameunder22dottedbox}
\end{figure}

\begin{table}[h!]
\begin{center}
\begin{tabular}{| c ||   c |  c |   c |  c |   c |  c |    c |   c |  c |   c  |}\hline
$a$                             & $0$       & $1$       & $2$      & $3$      &   $4$    & $5$      & $7$       & $8$    & $9$   & $10$ \\ \hline
Transf. $\alpha_{\dots}$        & $X$       & $X$       & $2$      & $1$      & $1$      & $1$      & $1$       & $1$    & $2$   & $1$ \\ \hline
\end{tabular}
\caption{Elimination of color $6$ from monochromatic crossings.}
\label{Ta:6alpha}
\end{center}
\end{table}


\begin{table}[h!]
\begin{center}
\begin{tabular}{| c ||   c |  c |   c |   c |  c |   c |  c |   c |  c |   c |}\hline
$a$                         & $0$ &  $1$ & $2$ & $3$   & $4$  & $5$  & $7$  & $8$  & $9$  & $10$ \\ \hline
Transf. $\beta_{\dots}$     & $X$ &  $X$ & $2$ & $1$   & $1$  & $1$  & $1$  & $2$  & $2$  & $1$   \\ \hline
\end{tabular}
\caption{Elimination of color $6$ from over-arcs of polichromatic crossings.}
\label{Ta:6beta}
\end{center}
\end{table}


\begin{table}[h!]
\begin{center}
\begin{tabular}{| c ||  c |||   c  |   c |  c |   c |  c |   c |  c |   c |  c | }\hline
$a=0$ &  $b$                        & $1$ & $2$ & $3$  & $4$  & $5$  & $7$  & $8$  & $9$  & $10$ \\ \hline
   &  Transf. $\gamma_{\dots}$      & $2$ & $X$ & $1$  & $2$  & $1$  & $2$  & $1$  & $X$  & $2$   \\ \hline \hline
$a=1$ &  $b$                        & $0$ & $2$ & $3$  & $4$   & $5$  & $7$  & $8$  & $9$  & $10$ \\ \hline
   &  Transf. $\gamma_{\dots}$      & $1$ & $X$ & $2$  & $5$  & $2$  & $1$  & $1$  & $X$  & $1$   \\ \hline \hline
$a=2$ &  $b$                        & $0$ & $1$ & $3$  & $4$   & $5$  & $7$  & $8$  & $9$  & $10$ \\ \hline
   &  Transf. $\gamma_{\dots}$      & $X$ & $X$ & $X$  & $X$   & $X$  & $X$  & $X$  & $X$  & $X$   \\ \hline \hline
$a=3$ &  $b$                        & $0$ & $1$ & $2$  & $4$  & $5$  & $7$  & $8$  & $9$  & $10$ \\ \hline
   &  Transf. $\gamma_{\dots}$      & $2$ & $1$ & $X$  & $1$  & $1$  & $4$  & $2$  & $X$  & $1$   \\ \hline \hline
$a=4$ &  $b$                        & $0$ & $1$ & $2$  & $3$   & $5$  & $7$  & $8$  & $9$  & $10$ \\ \hline
   &  Transf. $\gamma_{\dots}$      & $1$ & $6$ & $X$  & $1$   & $1$  & $1$  & $6$  & $X$  & $2$   \\ \hline \hline
$a=5$ &  $b$                        & $0$ & $1$ & $2$  & $3$  & $4$  & $7$  & $8$  & $9$  & $10$ \\ \hline
   &  Transf. $\gamma_{\dots}$      & $1$ & $1$ & $X$  & $1$  & $2$   & $1$  & $1$  & $X$  & $1$     \\ \hline \hline
$a=7$ &  $b$                        & $0$ & $1$ & $2$  & $3$  & $4$  & $5$  & $8$  & $9$  & $10$ \\ \hline
   &  Transf. $\gamma_{\dots}$      & $1$ & $1$ & $X$  & $3$  & $2$  & $1$  & $2$  & $X$  & $1$   \\ \hline \hline
$a=8$ &  $b$                        & $0$ & $1$ & $2$  & $3$   & $4$  & $5$  & $7$  & $9$  & $10$ \\ \hline
   &  Transf. $\gamma_{\dots}$      & $1$ & $1$ & $X$  & $1$   & $5$  & $2$  & $1$  & $X$  & $4$   \\ \hline \hline
$a=9$ &  $b$                        & $0$ & $1$ & $2$  & $3$   & $4$  & $5$  & $7$  & $8$  & $10$ \\ \hline
   &  Transf. $\gamma_{\dots}$      & $X$ & $X$ & $X$  & $X$   & $X$  & $X$  & $X$  & $X$  & $X$   \\ \hline \hline
$a=10$ &  $b$                       & $0$ & $1$ & $2$  & $3$   & $4$  & $5$  & $7$  & $8$  & $9$ \\ \hline
   &  Transf. $\gamma_{\dots}$      & $1$ & $2$ & $X$  & $1$   & $1$  & $1$  & $2$  & $3$  & $X$   \\ \hline \hline
\end{tabular}
\caption{Elimination of color $6$ from under-arcs joining crossings whose over-arcs bear distinct colors.}
\label{Ta:6gamma}
\end{center}
\end{table}


\begin{table}[h!]
\begin{center}
\begin{tabular}{| c ||   c |  c |   c |   c |  c |   c |  c |   c |  c |   c |}\hline
$a$                        & $0$ & $1$ & $2$ & $3$    & $4$  & $5$  & $7$  & $8$  & $9$  & $10$ \\ \hline
Transf. $\delta_{\dots}$   & $1$ & $2$ & $X$ & $1$    & $2$  & $1$  & $1$  & $3$  & $X$  & $1$   \\ \hline
\end{tabular}
\caption{Elimination of color $6$ from under-arcs joining crossings whose over-arcs bear the same color.}
\label{Ta:6delta}
\end{center}
\end{table}







\begin{table}[h!]
\begin{center}
\begin{tabular}{| c ||   c |  c |   c |  c |   c |  c |    c |   c |   c |   }\hline
$a$                          & $0$    &   $1$     & $2$     & $4$    & $5$    & $7$     & $8$   & $9$ & $10$ \\ \hline
Transf. $\alpha_{\dots}$     & $X$    &   $2$     & $1$     & $1$    & $1$    & $X$     & $X$   & $1$ & $1$ \\ \hline
\end{tabular}
\caption{Elimination of color $3$ from monochromatic crossings.}
\label{Ta:3alpha}
\end{center}
\end{table}


\begin{table}[h!]
\begin{center}
\begin{tabular}{| c ||   c |  c |   c |   c |  c |   c |  c |   c |  c |}\hline
$a$                         & $0$ &  $1$ & $2$ & $4$   & $5$  & $7$ & $8$  & $9$  & $10$ \\ \hline
Transf. $\beta_{\dots}$     & $X$ &  $2$ & $1$ & $2$   & $1$  & $X$ & $X$  & $1$  & $2$   \\ \hline
\end{tabular}
\caption{Elimination of color $3$ from over-arcs of polichromatic crossings.}
\label{Ta:3beta}
\end{center}
\end{table}


\begin{table}[h!]
\begin{center}
\begin{tabular}{| c ||  c |||   c  |   c |  c |   c |  c |   c |  c |   c |  }\hline
$a=0$ &  $b$                      & $1$ & $2$ & $4$  & $5$  & $7$  & $8$  & $9$  & $10$ \\ \hline
   &  Transf. $\gamma_{\dots}$    & $X$ & $2$ & $1$  & $2$  & $X$  & $1$  & $2$  & $2$   \\ \hline \hline
$a=1$ &  $b$                      & $0$ & $2$ & $4$  & $5$  & $7$  & $8$  & $9$  & $10$ \\ \hline
   &  Transf. $\gamma_{\dots}$    & $X$ & $X$ & $X$  & $X$  & $X$  & $X$  & $X$  & $X$   \\ \hline \hline
$a=2$ &  $b$                      & $0$ & $1$ & $4$  & $5$  & $7$  & $8$  & $9$  & $10$ \\ \hline
   &  Transf. $\gamma_{\dots}$    & $1$ & $X$ & $10$ & $2$  & $X$  & $1$  & $1$  & $1$   \\ \hline \hline
$a=4$ &  $b$                      & $0$ & $1$ & $2$  & $5$  & $7$  & $8$  & $9$  & $10$ \\ \hline
   &  Transf. $\gamma_{\dots}$    & $2$ & $X$ & $!10$ & $12$ & $X$  & $1$  & $2$  & $11$   \\ \hline \hline
$a=5$ &  $b$                      & $0$ & $1$ & $2$  & $4$  & $7$  & $8$  & $9$  & $10$ \\ \hline
   &  Transf. $\gamma_{\dots}$    & $1$ & $X$ & $1$  & $7$  & $X$  & $1$  & $1$  & $2$     \\ \hline \hline
$a=7$ &  $b$                      & $0$ & $1$ & $2$  & $4$  & $5$  & $8$  & $9$  & $10$ \\ \hline
   &  Transf. $\gamma_{\dots}$    & $X$ & $X$ & $X$  & $X$  & $X$  & $X$  & $X$  & $X$   \\ \hline \hline
$a=8$ &  $b$                      & $0$ & $1$ & $2$  & $4$  & $5$  & $7$  & $9$  & $10$ \\ \hline
   &  Transf. $\gamma_{\dots}$    & $2$ & $X$ & $1$  & $2$  & $2$  & $X$  & $1$  & $9$   \\ \hline \hline
$a=9$ &  $b$                      & $0$ & $1$ & $2$  & $4$  & $5$  & $7$  & $8$  & $10$ \\ \hline
   &  Transf. $\gamma_{\dots}$    & $1$ & $X$ & $2$  & $1$  & $2$  & $X$  & $1$  & $1$   \\ \hline \hline
$a=10$ &  $b$                     & $0$ & $1$ & $2$  & $4$  & $5$  & $7$  & $8$  & $9$ \\ \hline
   &  Transf. $\gamma_{\dots}$    & $1$ & $X$ & $2$  & $12$ & $1$  & $X$  & $!9$  & $2$   \\ \hline \hline
\end{tabular}
\caption{Elimination of color $3$ from under-arcs joining crossings whose over-arcs bear distinct colors ($a$ and $b$). The ``!'' means the values of the $a$ and $b$ parameters have to be interchanged.}
\label{Ta:3gamma}
\end{center}
\end{table}


\begin{table}[h!]
\begin{center}
\begin{tabular}{| c ||   c |   c |   c |  c |   c |  c |   c |  c |   c |}\hline
$a$                       & $0$ & $1$ & $2$ & $4$    & $5$   & $7$  & $8$  & $9$  & $10$ \\ \hline
Transf. $\delta_{\dots}$  & $1$ & $X$ & $2$ & $3$    & $5$   & $X$  & $1$  & $1$  & $4$   \\ \hline
\end{tabular}
\caption{Elimination of color $3$ from under-arcs joining crossings whose over-arcs bear the same color.}
\label{Ta:3delta}
\end{center}
\end{table}







\begin{table}[h!]
\begin{center}
\begin{tabular}{| c ||   c |  c |   c |  c |   c |  c |    c |   c |  c | }\hline
$a$                          & $0$        & $1$       & $2$       & $5$          & $7$       & $8$       & $9$      & $10$ \\ \hline
Transf. $\alpha_{\dots}$     & $1$        & $2$       & $X$       & $X$          & $1$       & $2$       & $X$      & $X$ \\ \hline
\end{tabular}
\caption{Elimination of color $4$ from monochromatic crossings.}
\label{Ta:4alpha}
\end{center}
\end{table}


\begin{table}[h!]
\begin{center}
\begin{tabular}{| c ||   c |  c |   c |   c |  c |   c |  c |   c |  c |}\hline
$a$                      & $0$ & $1$   & $2$   & $5$  & $7$ & $8$  & $9$  & $10$ \\ \hline
Transf. $\beta_{\dots}$  & $1$ & $2$   & $X$   & $X$  & $1$ & $2$  & $X$  & $X$   \\ \hline
\end{tabular}
\caption{Elimination of color $4$ from over-arcs of polichromatic crossings.}
\label{Ta:4beta}
\end{center}
\end{table}


\begin{table}[h!]
\begin{center}
\begin{tabular}{| c ||  c |||  c  |   c |  c |   c |  c |   c |  c |   c |  }\hline
$a=0$ &  $b$                      & $1$  & $2$ & $5$  & $7$  & $8$  & $9$  & $10$  \\ \hline
   &  Transf. $\gamma_{\dots}$    & $X$ & $13$  & $X$  & $2$  & $X$  & $2$  & $X$   \\ \hline \hline
$a=1$ &  $b$                      & $0$ & $2$  & $5$  & $7$  & $8$  & $9$  & $10$  \\ \hline
   &  Transf. $\gamma_{\dots}$    & $X$ & $X$  & $X$  & $X$  & $X$  & $X$  & $X$   \\ \hline \hline
$a=2$ &  $b$                      & $0$  & $1$  & $5$  & $7$  & $8$ & $9$  & $10$  \\ \hline
   &  Transf. $\gamma_{\dots}$    & $!13$ & $X$  & $X$  & $1$  & $X$ & $6$  & $X$   \\ \hline \hline
$a=5$ &  $b$                      & $0$ & $1$  & $2$  & $7$  & $8$  & $9$  & $10$  \\ \hline
   &  Transf. $\gamma_{\dots}$    & $X$ & $X$  & $X$  & $X$   & $X$  & $X$  & $X$     \\ \hline \hline
$a=7$ &  $b$                      & $0$ & $1$  & $2$  & $5$  & $8$  & $9$  & $10$  \\ \hline
   &  Transf. $\gamma_{\dots}$    & $1$ & $X$  & $2$  & $X$  & $X$  & $1$  & $X$   \\ \hline \hline
$a=8$ &  $b$                      & $0$  & $1$  & $2$  & $5$  & $7$  & $9$  & $10$  \\ \hline
   &  Transf. $\gamma_{\dots}$    & $X$ & $X$  & $X$   & $X$  & $X$  & $X$  & $X$   \\ \hline \hline
$a=9$ &  $b$                      & $0$ & $1$  & $2$  & $5$  & $7$  & $8$  & $10$  \\ \hline
   &  Transf. $\gamma_{\dots}$    & $1$ & $X$  & $5$   & $X$  & $2$  & $X$  & $X$   \\ \hline \hline
$a=10$ &  $b$                     & $0$ & $1$  & $2$ & $5$  & $7$  & $8$  & $9$  \\ \hline
   &  Transf. $\gamma_{\dots}$    & $X$ & $X$  & $X$  & $X$  & $X$  & $X$  & $X$   \\ \hline \hline
\end{tabular}
\caption{Elimination of color $4$ from under-arcs joining crossings whose over-arcs bear distinct colors ($a$ and $b$). The ``!'' means the values of the $a$ and $b$ parameters have to be interchanged.}
\label{Ta:4gamma}
\end{center}
\end{table}


\begin{table}[h!]
\begin{center}
\begin{tabular}{| c ||   c |   c |   c |  c |   c |  c |   c |  c |   c |}\hline
$a$                        & $0$ &  $1$ & $2$    & $5$  & $7$  & $8$  & $9$  & $10$ \\ \hline
Transf. $\delta_{\dots}$   & $1$ &  $X$ & $6$    & $X$  & $9$  & $X$  & $7$  & $X$   \\ \hline
\end{tabular}
\caption{Elimination of color $4$ from under-arcs joining crossings whose over-arcs bear the same color.}
\label{Ta:4delta}
\end{center}
\end{table}







\begin{table}[h!]
\begin{center}
\begin{tabular}{| c ||   c |  c |   c |  c |   c |  c |    c |   c |  }\hline
$a$                          & $0$      & $1$      & $2$    & $5$      & $7$   & $9$     & $10$          \\ \hline
Transf. $\alpha_{\dots}$     & $X$      & $1$      & $1$    & $X$      & $2$   & $1$     & $X$ \\ \hline
\end{tabular}
\caption{Elimination of color $8$ from monochromatic crossings.}
\label{Ta:8alpha}
\end{center}
\end{table}


\begin{table}[h!]
\begin{center}
\begin{tabular}{| c ||   c |  c |   c |   c |  c |   c |  c |   c |}\hline
$a$                      & $0$ & $1$ & $2$ & $5$  & $7$  & $9$   & $10$  \\ \hline
Transf. $\beta_{\dots}$  & $X$ & $1$ & $2$ & $X$ & $!3$  & $3$  & $X$   \\ \hline
\end{tabular}
\caption{Elimination of color $8$ from over-arcs of polichromatic crossings. The elimination of color $7$ via transformation $\beta_{3}$ assumes the change of variables $a=9, 2c-a=7$.}
\label{Ta:8beta}
\end{center}
\end{table}


\begin{table}[h!]
\begin{center}
\begin{tabular}{| c ||  c |||   c  |   c |  c |   c |  c |   c |  c |    }\hline
$a=0$ &  $b$                      & $1$  & $2$ & $5$ & $7$  & $9$  & $10$  \\ \hline
   &  Transf. $\gamma_{\dots}$    & $2$  & $9$& $2$ & $X$  & $2$  & $X$   \\ \hline \hline
$a=1$ &  $b$                      & $0$ & $2$ & $5$  & $7$  & $9$  & $10$  \\ \hline
   &  Transf. $\gamma_{\dots}$    & $1$ & $10$& $1$  & $X$  & $8$  & $X$   \\ \hline \hline
$a=2$ &  $b$                      & $0$  & $1$  & $5$  & $7$  & $9$  & $10$  \\ \hline
   &  Transf. $\gamma_{\dots}$    & $!9$ & $!10$ & $11$ & $X$  & $12$ & $X$   \\ \hline \hline
$a=5$ &  $b$                      & $0$ & $1$  &  $2$ & $7$ & $9$  & $10$  \\ \hline
   &  Transf. $\gamma_{\dots}$    & $1$ & $2$  &  $!11$ & $X$  & $1$  & $X$     \\ \hline \hline
$a=7$ &  $b$                      & $0$ & $1$  & $2$ & $5$  & $9$  & $10$  \\ \hline
   &  Transf. $\gamma_{\dots}$    & $X$  & $X$  & $X$  & $X$  & $X$  & $X$   \\ \hline \hline
$a=9$ &  $b$                      & $0$  & $1$  & $2$ & $5$  & $7$  & $10$  \\ \hline
   &  Transf. $\gamma_{\dots}$    & $1$  & $!8$ & $!12$ & $2$  & $X$  & $X$   \\ \hline \hline
$a=10$ &  $b$                     & $0$  & $1$  & $2$ & $5$  & $7$  & $9$\\ \hline
   &  Transf. $\gamma_{\dots}$    & $X$  & $X$  & $X$  & $X$  & $X$  & $X$   \\ \hline \hline
\end{tabular}
\caption{Elimination of color $8$ from under-arcs joining crossings whose over-arcs bear distinct colors ($a$ and $b$). The ``!'' means the values of the $a$ and $b$ parameters have to be interchanged.}
\label{Ta:8gamma}
\end{center}
\end{table}


\begin{table}[h!]
\begin{center}
\begin{tabular}{| c ||   c |   c |   c |  c |   c |  c |   c |  c | }\hline
$a$                          & $0$ & $1$  & $2$  & $5$  & $7$  & $9$    & $10$  \\ \hline
Transf. $\delta_{\dots}$     & $1$ & $8$ & $10$  & $4$  & $X$  & $11$   & $X$   \\ \hline
\end{tabular}
\caption{Elimination of color $8$ from under-arcs joining crossings whose over-arcs bear the same color.}
\label{Ta:8delta}
\end{center}
\end{table}

\section{Part III: Elimination of colors $9$ and $2$.}\label{subsec:92}






In this Section we eliminate colors $9$ and $2$. Since we have already eliminated six colors ($12, 11, 6, 3, 4$, and $8$) the most frequent symbol in the even numbered rows of the Tables in this Section is $X$. There are, also, four $\delta$ instances which cannot be dealt with the way we did in the preceding section. Instead we have to look into the colors at issue in order to produce adequate transformations. These instances are denoted $D_1, D_2, D_3,$ and $D_4$ in Tables \ref{Ta:9delta} and \ref{Ta:2delta} below and presented in the figures in this Section.

Since these $D_i$ transformations involve a different approach, a few words are in order here. $D_1$ and $D_2$ have to do with the $\delta$ instance of the elimination of color $9$, whereas $D_3$ and $D_4$ have to do with the $\delta$ instance of the elimination of color $2$. $D_1$ and $D_2$ ($D_3$ and $D_4$, respect.) are the very last cases to be resolved in the elimination of color $9$ (color $2$, respect.). Figure \ref{fig:d1} displays transformation $D_1$. It has to do with the elimination of color $9$ from an under-arc between two over-arcs colored both with color $5$. Transformation $D_1$ accomplishes this by reducing the problem to the elimination of color $9$ from an under-arc between two over-arcs, one colored $5$ and the other colored $7$. This situation has been resolved before - see Table \ref{Ta:9gamma}, Transformation $\gamma_{14}$ - so the elimination of color $9$ from an under-arc between two over-arcs colored with $5$ is accomplished.

Now for transformation $D_2$, which realizes the elimination of color $9$ from an under-arc between two over-arcs colored $7$. We start by realizing what the possibilities are for the over-arcs colored $7$. These over-arcs have to eventually end up at a polichromatic crossing for otherwise there would be a split component colored with $7$, and the link would have $0$ determinant which contradicts our assumptions. Since the only colors available now are $0, 1, 2, 5, 7, 10$ the possibilities for the triplets $\{ a, b, c \}$ from this set that realize $2b=a+c$ mod $13$ are displayed in Table \ref{Ta:tripletsfrom0125710}. There are thus two possibilities for an over-arc colored with $7$. Either the $7$ ends up at a crossing whose over-arc is colored $2$ and the other under-arc is colored $10$; or the over-arc is colored $10$ and the other under-arc is colored $0$. In Figure \ref{fig:d20} we show that other possibilities for colors at under-arcs compliant with the $7$ on the over-arc before it ends up at a polichromatic crossing do not impede the progress of color $10$. The role of color $10$ in the elimination of color $9$ is shown in Figures \ref{fig:d2} and \ref{fig:d2bis}. They show the elimination of color $9$ from an under-arc between two over-arcs colored with $7$ is accomplished.

Now for transformation $D_3$, which realizes the elimination of color $2$ from an under-arc between two over-arcs colored $1$. This over-arc colored $1$ has to eventually end up at a polichromatic crossing. There is only one possibility for this polichromatic crossing: its over-arc is colored $7$ and the other under-arc is colored $0$, see Table \ref{Ta:tripletsfrom015710}. Figure \ref{fig:d30} shows us that other crossings that the $1$ may be an over-arc to, do not impede the progression of the $7$ to a convenient neighborhood of the $2$ we would like to eliminate. Then Figure \ref{fig:d3} shows how to eliminate color $2$ from the situation at issue.

Finally, for transformation $D_4$, which realizes the elimination of color $2$ from an under-arc between two over-arcs colored $10$. We argue as follows. There has to be an arc colored $1$ somewhere in the diagram under study. By performing Reidemeister type II moves we bring the $1$ to a convenient neighborhood of the $2$ we want to eliminate. We do eliminate the $2$ with the help of the $1$ since the colors available besides $1$ are $0, 5, 7, 10$. As a matter of fact, we can see in Table \ref{Ta:tripletsfrom015710} that a $1$ going over  $5$ produces a $10$ and a $1$ going over a $10$ produces a $5$ and these colors are all admissible. A $1$ going over a $0$ gives rise to a $2$ and we saw in Figure \ref{fig:d3} how to deal with this situation. Finally, a $1$ going over a $7$ gives rise to an $8$ and we deal with this situation using Transformation $\delta_8$, see Figure \ref{fig:sameunder18} and Table \ref{Ta:8delta}, with $a=1$ and $c=8$. When we list the linear expressions corresponding to Transformation $\delta_8$ and evaluated at $a=1$ and $c=8$ we obtain three colors outside the set $\{ 0, 1, 5, 7, 10 \}$:
\[
2c-a = 2 \qquad \qquad 12a-11c = 2 \qquad \qquad 3c-2a = 9
\]
But this $2c-a=2$ in the final diagram of Transformation $\delta_8$ corresponds to an under-arc colored $2$ between two over-arcs colored $1$ - and this has been resolved in Figure \ref{fig:d3} with Transformation $D_3$. As for $12a-11c=2$, it corresponds to an under-arc colored $2$ between an over-arc colored $1$ and an over-arc colored $10$ and this has been dealt with (see Table \ref{Ta:2gamma}). Finally, $3c-2a=9$ corresponds to an over-arc colored $9$ between two over-arcs colored $7$ and this has also been dealt with (see Figure \ref{fig:d2}). Then Figure \ref{fig:d4new} shows how to proceed in the elimination of color $2$ from an under-arc between two over-arcs colored $10$. This concludes the proof of Theorem \ref{thm:main}.


\begin{table}[h!]
\begin{center}
\begin{tabular}{| c ||   c |  c |   c |  c |   c |  c |    c |    }\hline
$a$                          & $0$   & $1$       & $2$       & $5$      & $7$       & $10$          \\ \hline
Transf. $\alpha_{\dots}$     & $2$   & $X$       & $X$       & $1$      & $X$       & $X$ \\ \hline
\end{tabular}
\caption{Elimination of color $9$ from monochromatic crossings.}
\label{Ta:9alpha}
\end{center}
\end{table}


\begin{table}[h!]
\begin{center}
\begin{tabular}{| c ||   c |  c |   c |   c |  c |   c |  c |  }\hline
$a$                         & $0$   & $1$    & $2$   & $5$  & $7$  & $10$      \\ \hline
Transf. $\beta_{\dots}$     & $2$   & $X$    & $X$   & $1$  & $X$  & $X$   \\ \hline
\end{tabular}
\caption{Elimination of color $9$ from over-arcs of polichromatic crossings.}
\label{Ta:9beta}
\end{center}
\end{table}


\begin{table}[h!]
\begin{center}
\begin{tabular}{| c ||  c |||   c  |   c |  c |   c |  c |   c |     }\hline
$a=0$ &  $b$                     & $1$    & $2$   & $5$  & $7$  & $10$     \\ \hline
   &  Transf. $\gamma_{\dots}$   & $X$ & $X$   & $X$  & $X$  & $X$   \\ \hline \hline
$a=1$ &  $b$                     & $0$    & $2$   & $5$  & $7$  & $10$     \\ \hline
   &  Transf. $\gamma_{\dots}$   & $X$ & $X$  & $X$  & $X$  & $X$   \\ \hline \hline
$a=2$ &  $b$                     & $0$   & $1$    & $5$  & $7$  & $10$     \\ \hline
   &  Transf. $\gamma_{\dots}$   & $X$   & $X$ & $X$  & $X$  & $X$   \\ \hline \hline
$a=5$ &  $b$                     & $0$   & $1$    & $2$   & $7$  & $10$     \\ \hline
   &  Transf. $\gamma_{\dots}$   & $X$ & $X$   & $X$  & $14$ & $X$     \\ \hline \hline
$a=7$ &  $b$                     & $0$   & $1$    & $2$   & $5$  & $10$     \\ \hline
   &  Transf. $\gamma_{\dots}$   & $X$  & $X$  & $X$  & $!14$  & $X$   \\ \hline \hline
$a=10$ &  $b$                    & $0$   & $1$    & $2$   & $5$  & $7$    \\ \hline
   &  Transf. $\gamma_{\dots}$   & $X$  & $X$  & $X$  & $X$  & $X$   \\ \hline \hline
\end{tabular}
\caption{Elimination of color $9$ from under-arcs joining crossings whose over-arcs bear distinct colors ($a$ and $b$). The ``!'' means the values of the $a$ and $b$ parameters have to be interchanged.}
\label{Ta:9gamma}
\end{center}
\end{table}


\begin{table}[h!]
\begin{center}
\begin{tabular}{| c ||   c |   c |   c |  c |   c |  c |   c |  }\hline
$a$                        & $0$   & $1$    & $2$   & $5$    & $7$    & $10$     \\ \hline
Transf. $\delta_{\dots}$   & $X$   & $X$    & $X$   & $D_1$  & $D_2$  & $X$   \\ \hline
\end{tabular}
\caption{Elimination of color $9$ from under-arcs joining crossings whose over-arcs bear the same color.}
\label{Ta:9delta}
\end{center}
\end{table}

\begin{figure}[!ht]
	\psfrag{0}{\huge$0$}
	\psfrag{1}{\huge$1$}
	\psfrag{2}{\huge$2$}
	\psfrag{5}{\huge$5$}
	\psfrag{7}{\huge$7$}
	\psfrag{9}{\huge$9$}
	\psfrag{10}{\huge$10$}
	\psfrag{13}{\huge$\mathbf{13}$}
	 \centerline{\scalebox{.4}{\includegraphics{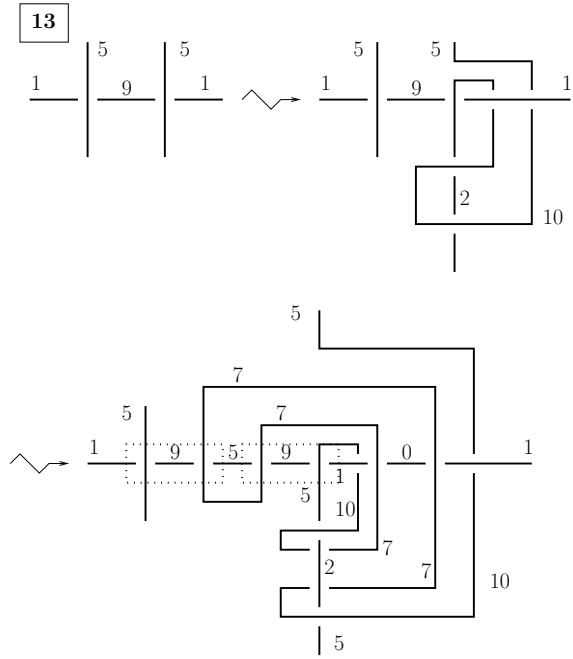}}}
	\caption{Transformation $D_{1}$. The issues in the dotted boxes are dealt with with transformation $\gamma_{14}$, see Table \ref{Ta:9gamma}.}\label{fig:d1}
\end{figure}

\begin{table}[h!]
\begin{center}
\begin{tabular}{| c ||   c |  c |   c |  c |   c |  c |    c |    }\hline
color on over-arc                          & $0$   & $1$       & $2$       & $5$      & $7$       & $10$          \\ \hline
colors on under-arcs     &    & $\{ 0, 2  \}$       & $\{ 7, 10  \}$       & $\{ 0, 10  \}$   & $\{ 0, 1 \}$      & $\{ 0, 7  \}$       \\
                         &    & $\{ 5, 10  \}$       &                      &                  &                   & $\{ 2, 5  \}$       \\ \hline
\end{tabular}
\caption{For each color $b$ on the top row, the duplet(s) under it displays the colors $\{ a, c \}$ from $\{ 0, 1, 2, 5, 7, 10 \}$ on the under-arcs that satisfy $2b=a+c$, mod $13$, non-trivially.}
\label{Ta:tripletsfrom0125710}
\end{center}
\end{table}

\begin{figure}[!ht]
	\psfrag{0}{\huge$0$}
	\psfrag{1}{\huge$1$}
	\psfrag{2}{\huge$2$}
	\psfrag{5}{\huge$5$}
	\psfrag{7}{\huge$7$}
	\psfrag{9}{\huge$9$}
	\psfrag{10}{\huge$10$}
	\psfrag{11}{\huge$11$}
	\psfrag{8}{\huge$8$}
	\psfrag{12}{\huge$12$}
	\psfrag{13}{\huge$\mathbf{13}$}
	 \centerline{\scalebox{.4}{\includegraphics{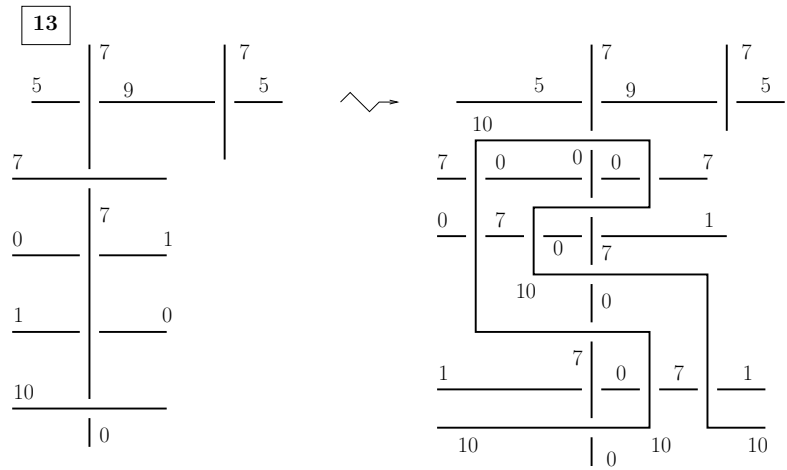}}}
	\caption{Preliminary considerations for Transformation $D_{2}$ below. These preliminaries allow us to disregard some complications below by showing that the $10$ can progress all the way up to a convenient neighborhood of the $9$.}\label{fig:d20}
\end{figure}

\begin{figure}[!ht]
	\psfrag{0}{\huge$0$}
	\psfrag{1}{\huge$1$}
	\psfrag{2}{\huge$2$}
	\psfrag{5}{\huge$5$}
	\psfrag{7}{\huge$7$}
	\psfrag{9}{\huge$9$}
	\psfrag{10}{\huge$10$}
	\psfrag{11}{\huge$11$}
	\psfrag{8}{\huge$8$}
	\psfrag{12}{\huge$12$}
	\psfrag{13}{\huge$\mathbf{13}$}
	 \centerline{\scalebox{.4}{\includegraphics{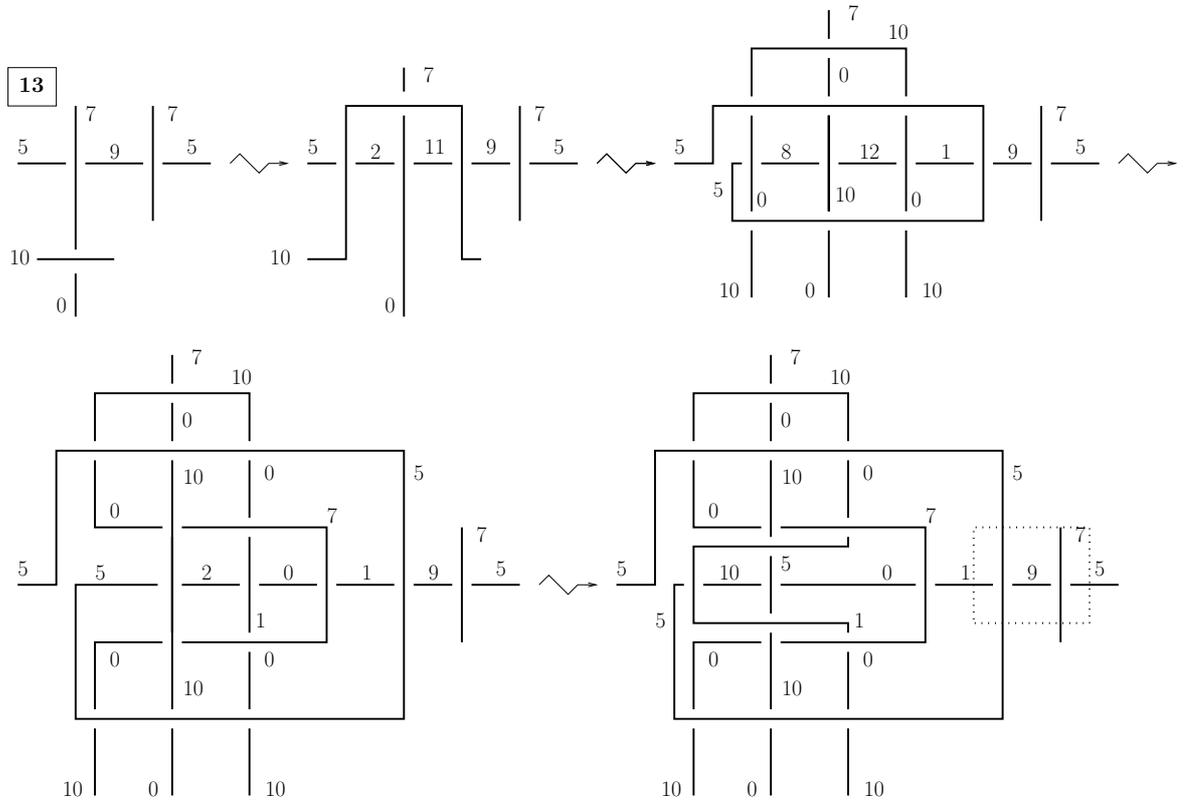}}}
	\caption{Transformation $D_{2}$, part $I$: the over-arc on the top left bearing color $7$ ends up at a crossing whose over-arc bears color $10$. The issue in the dotted box is  dealt with with transformation $\gamma_{14}$, see Table \ref{Ta:9gamma}. The last step will be useful for considerations in the sequel when removing $2$ from the list of colors, see Figures \ref{fig:d30} and \ref{fig:d40+}.}\label{fig:d2}
\end{figure}

\begin{figure}[!ht]
	\psfrag{0}{\huge$0$}
	\psfrag{1}{\huge$1$}
	\psfrag{2}{\huge$2$}
	\psfrag{5}{\huge$5$}
	\psfrag{7}{\huge$7$}
	\psfrag{9}{\huge$9$}
	\psfrag{10}{\huge$10$}
	\psfrag{11}{\huge$11$}
	\psfrag{8}{\huge$8$}
	\psfrag{12}{\huge$12$}
	\psfrag{13}{\huge$\mathbf{13}$}
	 \centerline{\scalebox{.4}{\includegraphics{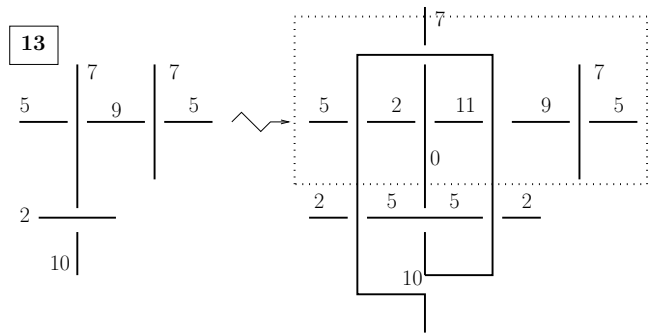}}}
	\caption{Transformation $D_{2}$, part $II$: the over-arc on the top left bearing color $7$ ends up at a crossing whose over-arc bears color $2$. The issue in the dotted box is dealt with with the transformation from Figure \ref{fig:d2}.}\label{fig:d2bis}
\end{figure}







\begin{table}[h!]
\begin{center}
\begin{tabular}{| c ||   c |  c |   c |  c |   c |  c |     }\hline
$a$                       & $0$   & $1$   & $5$  & $7$  & $10$     \\ \hline
Transf. $\alpha_{\dots}$  & $X$   & $X$   & $X$  & $2$  & $1$ \\ \hline
\end{tabular}
\caption{Elimination of color $2$ from monochromatic crossings.}
\label{Ta:2alpha}
\end{center}
\end{table}


\begin{table}[h!]
\begin{center}
\begin{tabular}{| c ||   c |  c |   c |   c |  c |   c |  c |   c |}\hline
$a$                      & $0$   & $1$ & $5$  & $7$  & $10$     \\ \hline
Transf. $\beta_{\dots}$  & $X$   & $X$ & $X$  & $2$  & $1$   \\ \hline
\end{tabular}
\caption{Elimination of color $2$ from over-arcs of polichromatic crossings.}
\label{Ta:2beta}
\end{center}
\end{table}


\begin{table}[h!]
\begin{center}
\begin{tabular}{| c ||  c |||   c  |   c |  c |   c |  c |    } \hline
$a=0$ &  $b$                     & $1$  & $5$  & $7$  & $10$    \\ \hline
   &  Transf. $\gamma_{\dots}$   & $X$  & $X$  & $X$  & $X$   \\ \hline \hline
$a=1$ &  $b$                     & $0$  & $5$  & $7$  & $10$    \\ \hline
   &  Transf. $\gamma_{\dots}$   & $X$  & $X$  & $X$  & $1$   \\ \hline \hline
$a=5$ &  $b$                     & $0$  & $1$  & $7$  & $10$    \\ \hline
   &  Transf. $\gamma_{\dots}$   & $X$  & $X$  & $X$  & $X$     \\ \hline \hline
$a=7$ &  $b$                     & $0$  & $1$  & $5$  & $10$    \\ \hline
   &  Transf. $\gamma_{\dots}$   & $X$  & $X$  & $X$  & $X$   \\ \hline \hline
$a=10$ &  $b$                    & $0$  & $1$  & $5$  & $7$   \\ \hline
   &  Transf. $\gamma_{\dots}$   & $X$  & $2$  & $X$  & $X$   \\ \hline
\end{tabular}
\caption{Elimination of color $2$ from under-arcs joining crossings whose over-arcs bear distinct colors ($a$ and $b$).}
\label{Ta:2gamma}
\end{center}
\end{table}


\begin{table}[h!]
\begin{center}
\begin{tabular}{| c ||   c |   c |   c |  c |   c |  c |   c |  c | }\hline
$a$                       & $0$ & $1$    & $5$  & $7$  & $10$     \\ \hline
Transf. $\delta_{\dots}$  & $X$ & $D_3$  & $X$  & $X$  & $D_4$   \\ \hline
\end{tabular}
\caption{Elimination of color $2$ from under-arcs joining crossings whose over-arcs bear the same color.}
\label{Ta:2delta}
\end{center}
\end{table}

\begin{table}[h!]
\begin{center}
\begin{tabular}{| c ||   c |  c |   c |   c |  c |   c |  c |   c |}\hline
color on over-arc                      & $0$   & $1$ & $5$  & $7$  & $10$     \\ \hline
colors on under-arc  &    & $\{ 5, 10 \}$ & $\{ 0, 10 \}$  & $\{ 0, 1 \}$  & $\{ 0, 7 \}$   \\ \hline
\end{tabular}
\caption{For each color $b$ on the top row, the duplet under it displays the colors $\{ a, c \}$ from $\{ 0, 1, 5, 7, 10 \}$ on the under-arcs that satisfy $2b=a+c$, mod $13$, non-trivially.}
\label{Ta:tripletsfrom015710}
\end{center}
\end{table}

\begin{figure}[!ht]
	\psfrag{0}{\huge$0$}
	\psfrag{1}{\huge$1$}
	\psfrag{2}{\huge$2$}
	\psfrag{5}{\huge$5$}
	\psfrag{7}{\huge$7$}
	\psfrag{9}{\huge$9$}
	\psfrag{10}{\huge$10$}
	\psfrag{11}{\huge$11$}
	\psfrag{8}{\huge$8$}
	\psfrag{12}{\huge$12$}
	\psfrag{13}{\huge$\mathbf{13}$}
	 \centerline{\scalebox{.4}{\includegraphics{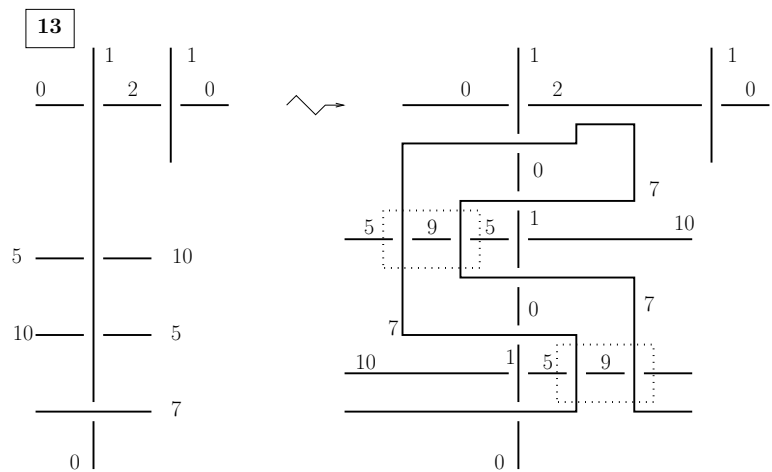}}}
	\caption{Preliminaries to Transformation $D_{3}$. These preliminaries allow us to disregard some complications below by showing that the $7$ can progress all the way up to a convenient neighborhood of the $2$. Note that the issues in the dotted boxes have been resolved before without resorting to color $2$, see Figure \ref{fig:d2}. Note also that Transformation $\gamma_{14}$ (Figure \ref{fig:diffunder76}) for $a=5, b=7, c=9$ only involves $0, 1, 5, 7, 10$.  }\label{fig:d30}
\end{figure}

\begin{figure}[!ht]
	\psfrag{0}{\huge$0$}
	\psfrag{1}{\huge$1$}
	\psfrag{2}{\huge$2$}
	\psfrag{5}{\huge$5$}
	\psfrag{7}{\huge$7$}
	\psfrag{9}{\huge$9$}
	\psfrag{10}{\huge$10$}
	\psfrag{11}{\huge$11$}
	\psfrag{8}{\huge$8$}
	\psfrag{12}{\huge$12$}
	\psfrag{13}{\huge$\mathbf{13}$}
	 \centerline{\scalebox{.4}{\includegraphics{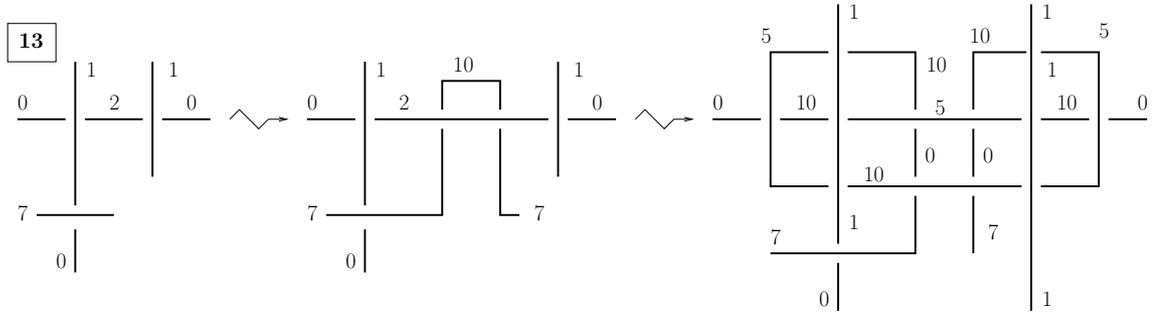}}}
	\caption{Transformation $D_{3}$.}\label{fig:d3}
\end{figure}

\begin{figure}[!ht]
	\psfrag{0}{\huge$0$}
	\psfrag{1}{\huge$1$}
	\psfrag{2}{\huge$2$}
	\psfrag{5}{\huge$5$}
	\psfrag{7}{\huge$7$}
	\psfrag{9}{\huge$9$}
	\psfrag{10}{\huge$10$}
	\psfrag{11}{\huge$11$}
	\psfrag{8}{\huge$8$}
	\psfrag{12}{\huge$12$}
	\psfrag{13}{\huge$\mathbf{13}$}
	 \centerline{\scalebox{.4}{\includegraphics{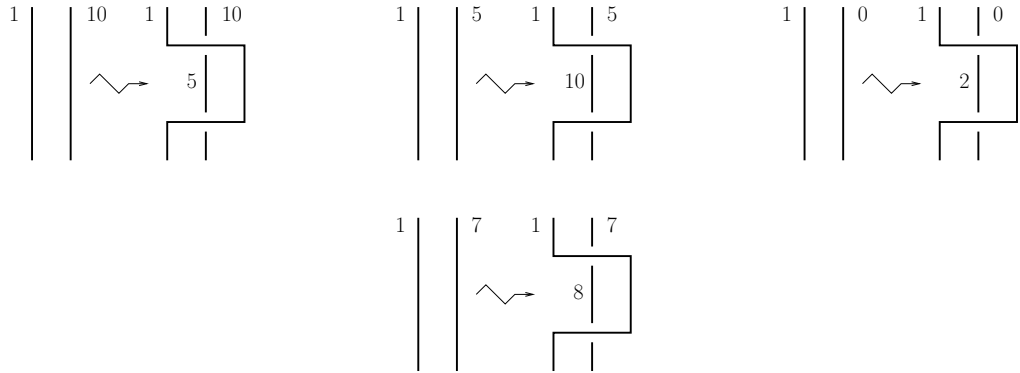}}}
	\caption{Preliminaries to Transformation $D_{4}$. In the first row, there are three instances of an arc colored $1$ moving past arcs colored $5$, $10$, and $0$. For the first two, the new color that shows up belongs to the set of colors available, $\{ 0, 1, 5, 7, 10  \}$. For the last one, transformation $D_3$ in Figure \ref{fig:d3} shows how to eliminate color $2$ without resorting to colors outside the set $\{ 0, 1, 5, 7, 10  \}$. In the second row the arc colored $1$ moves past the arc colored $7$ producing an $8$. This $8$ can be eliminated via Transformation $\delta_8$ and then some more considerations, as detailed in Figure \ref{fig:d40+}.}\label{fig:d40}
\end{figure}

\clearpage

\begin{figure}[!ht]
	\psfrag{a}{\huge$a$}
	\psfrag{2c-a}{\huge$2c-a$}
	\psfrag{3a-2c}{\huge$3a-2c$}
	\psfrag{11c-10a}{\huge$11c-10a$}
	\psfrag{2a-c}{\huge$2a-c$}
	\psfrag{3c-2a}{\huge$3c-2a$}
	\psfrag{6a-5c}{\huge$6a-5c$}
	\psfrag{12a-11c}{\huge$12a-11c$}
	\psfrag{9a-8c}{\huge$9a-8c$}
	\psfrag{0}{\huge$0$}
	\psfrag{1}{\huge$1$}
	\psfrag{2}{\huge$2$}
	\psfrag{5}{\huge$5$}
	\psfrag{7}{\huge$7$}
	\psfrag{9}{\huge$9$}
	\psfrag{10}{\huge$10$}
	\psfrag{11}{\huge$11$}
	\psfrag{8}{\huge$8$}
	\psfrag{12}{\huge$12$}
	\psfrag{13}{\huge$\mathbf{13}$}
	 \centerline{\scalebox{.4}{\includegraphics{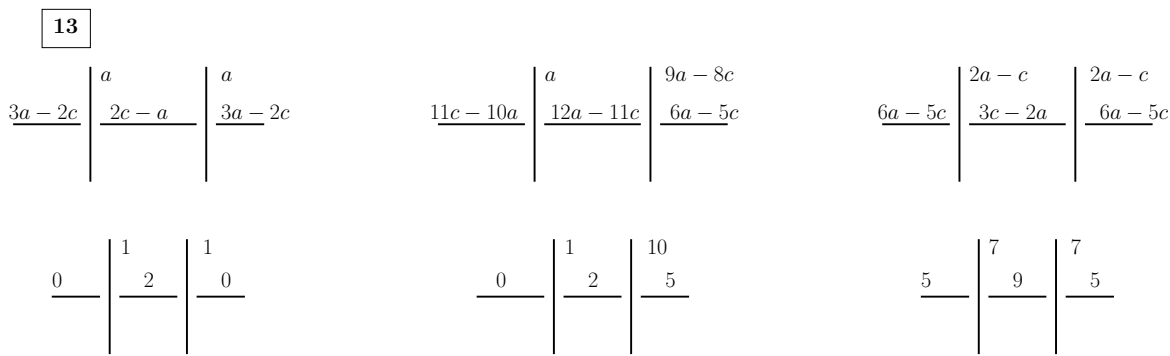}}}
	\caption{Preliminaries to Transformation $D_{4}$: resolving the $8$ produced by having an arc colored $1$ going over an arc colored $7$. The three columns in this Figure show the three instances where Transformation $\delta_8$, applied with $a=1$ and $c=8$, yield colors outside the set $\{ 0, 1, 5, 7, 10 \}$, namely $2$ and $9$. The instance corresponding to the leftmost column is resolved as depicted in Figure \ref{fig:d3}. The instance corresponding to the middle column is resolved with Transformation $\gamma_1$, see Table \ref{Ta:2gamma}. Finally, the instance corresponding to rightmost column is resolved as depicted in Figure \ref{fig:d2}. We remark that there still has to be used Transformation $\gamma_{14}$, see Figure \ref{fig:diffunder76}, with $a=5, b=7, c=9$ which eliminates color $9$ resorting only to colors from $\{ 0, 1, 5, 7, 10  \}$.}\label{fig:d40+}
\end{figure}

\begin{figure}[!ht]
	\psfrag{0}{\huge$0$}
	\psfrag{1}{\huge$1$}
	\psfrag{2}{\huge$2$}
	\psfrag{5}{\huge$5$}
	\psfrag{7}{\huge$7$}
	\psfrag{9}{\huge$9$}
	\psfrag{10}{\huge$10$}
	\psfrag{11}{\huge$11$}
	\psfrag{8}{\huge$8$}
	\psfrag{12}{\huge$12$}
	\psfrag{13}{\huge$\mathbf{13}$}
	 \centerline{\scalebox{.4}{\includegraphics{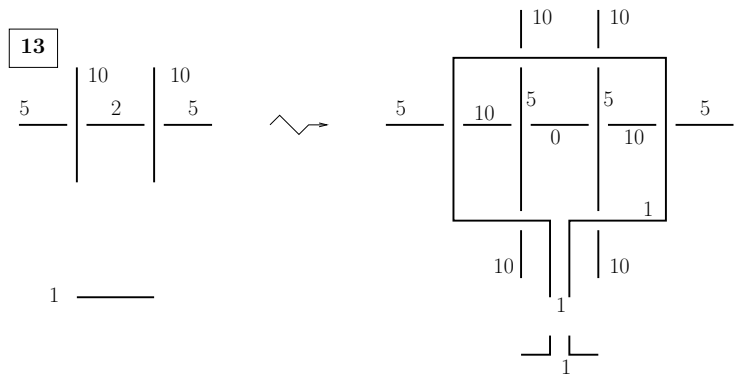}}}
	\caption{Transformation $D_{4}$.}\label{fig:d4new}
\end{figure}

\end{document}